\documentclass[12pt]{amsart}
\pdfoutput=1
\usepackage{latexsym, amsmath, amssymb, amsthm, mathrsfs, bm}
\usepackage[mathscr]{eucal}
\usepackage{paralist, array}
\usepackage{tikz}
\usetikzlibrary{arrows,shapes,positioning,calc,patterns,cd}
\usepackage{graphicx}
\usepackage[alphabetic]{amsrefs}
\usepackage[T1]{fontenc}
\usepackage{mathptmx}
\usepackage{microtype}

\usepackage[colorlinks=true,linkcolor=blue,urlcolor=blue, citecolor=blue]%
  {hyperref}
\usepackage[centering, includeheadfoot, hmargin=1.0in, tmargin=1.0in, 
  bmargin=1in, headheight=7pt]{geometry}
\newtheorem{lemma}{Lemma}[section]
\newtheorem{theorem}[lemma]{Theorem}
\newtheorem{corollary}[lemma]{Corollary}
\newtheorem{proposition}[lemma]{Proposition}
\theoremstyle{definition}
\newtheorem{definition}[lemma]{Definition}
\newtheorem{remark}[lemma]{Remark}
\newtheorem{examplex}[lemma]{Example}
\newenvironment{example}
  {\pushQED{\qed}\examplex}
  {\popQED\endexamplex}
\newtheorem{question}[lemma]{Question}
\newtheorem{algorithm}[lemma]{Algorithm}
\newcommand{\define}[1]{{\bfseries\itshape #1}}
\newcommand{\relphantom}[1]{\mathrel{\phantom{#1}}}
\newcommand{\abs}[1]{\ensuremath{\left| #1 \right|}}
\newcommand{\ideal}[1]{\ensuremath{\left\langle #1 \right\rangle}}
\renewcommand{\gets}{\longleftarrow}
\renewcommand{\geq}{\geqslant}
\renewcommand{\leq}{\leqslant}
\newcommand{\coloneq}{\ensuremath{\mathrel{\mathop :}=}}
\newcommand{\kk}{\ensuremath{\Bbbk}} 
\newcommand{\NN}{\ensuremath{\mathbb{N}}}
\newcommand{\PP}{\ensuremath{\mathbb{P}}}

\newcommand{\ZZ}{\ensuremath{\mathbb{Z}}} 
\newcommand{\bba}{\ensuremath{\bm{a}}}
\newcommand{\bbb}{\ensuremath{\bm{b}}}
\newcommand{\bbc}{\ensuremath{\bm{c}}}
\newcommand{\bbd}{\ensuremath{\bm{d}}}
\newcommand{\bbe}{\ensuremath{\bm{e}}}
\newcommand{\bbm}{\ensuremath{\bm{m}}}
\newcommand{\bbn}{\ensuremath{\bm{n}}}
\newcommand{\bbp}{\ensuremath{\bm{p}}}
\newcommand{\bbq}{\ensuremath{\bm{q}}}
\newcommand{\bbu}{\ensuremath{\bm{u}}}
\newcommand{\bbv}{\ensuremath{\bm{v}}}

\newcommand{\bF}{\ensuremath{{F}}}

\newcommand{\boldzero}{\ensuremath{\bm{0}}}
\newcommand{\boldone}{\ensuremath{\bm{1}}}
\newcommand{\cE}{\ensuremath{\mathcal{E}}}
\newcommand{\cF}{\ensuremath{\mathcal{F}}}
\newcommand{\cK}{\ensuremath{\mathcal{K}}}
\newcommand{\cO}{\ensuremath{\mathcal{O}}}
\newcommand{\cT}{\ensuremath{\mathcal{T}}}
\newcommand{\cHom}{\ensuremath{\!\mathcal{H}\!\!\textit{om}}}
\DeclareMathOperator{\codim}{codim}
\DeclareMathOperator{\coker}{Coker}

\DeclareMathOperator{\HH}{H}
\DeclareMathOperator{\HHH}{\mathbb{H}}
\DeclareMathOperator{\Image}{Im}
\DeclareMathOperator{\Ker}{Ker}
\DeclareMathOperator{\pdim}{pdim}
\DeclareMathOperator{\Pic}{Pic}

\DeclareMathOperator{\pr}{p}
\DeclareMathOperator{\rank}{rank}
\DeclareMathOperator{\reg}{reg}
\DeclareMathOperator{\Spec}{Spec}

\DeclareMathOperator{\variety}{V}

\makeatletter
\def\bbordermatrix#1{\begingroup \m@th
  \@tempdima 4.75\p@
  \setbox\z@\vbox{%
    \def\cr{\crcr\noalign{\kern2\p@\global\let\cr\endline}}%
    \ialign{$##$\hfil\kern2\p@\kern\@tempdima&\thinspace\hfil$##$\hfil
      &&\quad\hfil$##$\hfil\crcr
      \omit\strut\hfil\crcr\noalign{\kern-\baselineskip}%
      #1\crcr\omit\strut\cr}}%
  \setbox\tw@\vbox{\unvcopy\z@\global\setbox\@ne\lastbox}%
  \setbox\tw@\hbox{\unhbox\@ne\unskip\global\setbox\@ne\lastbox}%
  \setbox\tw@\hbox{$\kern\wd\@ne\kern-\@tempdima\left[\kern-\wd\@ne
    \global\setbox\@ne\vbox{\box\@ne\kern2\p@}%
    \vcenter{\kern-\ht\@ne\unvbox\z@\kern-\baselineskip}\,\right]$}%
  \null\;\vbox{\kern\ht\@ne\box\tw@}\endgroup}
\makeatother

\begin{document}

\vspace*{3.0em}

\title[Virtual resolutions]%
  {Virtual resolutions for a product of projective spaces}

\author[C.~Berkesch]{Christine Berkesch}
\address{Christine Berkesch: School of Mathematics, University of
  Minnesota, Minneapolis, Minnesota, 55455, United States of America;
  {\normalfont \texttt{cberkesc@math.umn.edu}}}

\author[D.~Erman]{Daniel Erman} 
\address{Daniel Erman: Department of Mathematics, University of Wisconsin,
  Madison, Wisconsin, 53706, United States of America; {\normalfont
    \texttt{derman@math.wisc.edu}}}

\author[G.G.~Smith]{Gregory G.{} Smith}
\address{Gregory G.{} Smith: Department of Mathematics \& Statistics, Queen's
  University, Kingston, Ontario, K7L 3N6, Canada; {\normalfont
    \texttt{ggsmith@mast.queensu.ca}}} 

\thanks{CB was partially supported by the NSF Grant DMS-1440537, DE was
  partially supported by the NSF Grants DMS-1302057 and DMS-1601619, and GGS
  was partially supported by the NSERC}


%\date(12 March 2019}

\begin{abstract}
  Syzygies capture intricate geometric properties of a subvariety in
  projective space.  However, when the ambient space is a product of
  projective spaces or a more general smooth projective toric variety, minimal
  free resolutions over the Cox ring are too long and contain many
  geometrically superfluous summands.  In this paper, we construct some much
  shorter free complexes that better encode the geometry.
\end{abstract}

\maketitle

%%%%%%%%%%%%%%%%%%%%%%%%%%%%%%%%%%%%%%%%%%%%%%%%%%%%%%%%%%%%%%%%%%%%%%%%%%%%%%
\setcounter{section}{1}
% \section{Introduction}

\noindent
The geometric and algebraic sources of locally-free resolutions have
complementary advantages.  To see the differences, consider a smooth
projective toric variety $X$ together with its $\Pic(X)$-graded Cox ring $S$.
The local version of the Hilbert Syzygy Theorem implies that any coherent
$\cO_X^{}$-module admits a locally-free resolution of length at most $\dim X$;
see Exercise~III.6.9 in \cite{Hartshorne}.  The global version of the Hilbert
Syzygy Theorem implies that every saturated module over the polynomial ring
$S$ has a minimal free resolution of length at most $\dim S - 1$, so any
coherent $\cO_X^{}$-module has a locally-free resolution of the same length;
see Proposition~3.1 in \cite{cox}.  Unlike the geometric approach, this
algebraic method only involves vector bundles that are a direct sum of line
bundles. When $X$ is projective space, these geometric and algebraic
constructions usually coincide.  However, when the Picard number of $X$ is
greater than $1$, the locally-free resolutions arising from the minimal free
resolution of an $S$-module are longer, and typically much longer, than their
geometric counterparts.

To enjoy the best of both worlds, we focus on a more flexible algebraic source
for locally-free resolutions.  The following definition, beyond providing
concise terminology, highlights this source.

\begin{definition}
  \label{def:splendid}
  A free complex $F \coloneq [F_0 \gets F_1 \gets F_2 \gets \dotsb]$ of
  $\Pic(X)$-graded $S$-modules is called a \define{virtual resolution} of a
  $\Pic(X)$-graded $S$-module $M$ if the corresponding complex $\widetilde{F}$
  of vector bundles on $X$ is a locally-free resolution of the sheaf
  $\widetilde{M}$.
\end{definition}

\noindent
In other words, a virtual resolution is a free complex of $S$-modules whose
higher homology groups are supported on the irrelevant ideal of $X$.  The
benefits of allowing a limited amount of homology are already present in other
parts of commutative algebra including almost ring
theory~\cite{almost-ring-theory}, where one accepts homology annihilated by a
given idempotent ideal, and phantom homology~\cite{phantom}, where one admits
cycles that are in the tight closure of the boundaries.  In this paper, we
describe a few different, and generally incomparable, processess for creating
virtual resolutions.

For projective space, minimal free resolutions are important in the study of
points~\cites{GGP, EP}, curves~\cites{V, EL}, surfaces~\cites{GP, DS}, and
moduli spaces~\cites{farkas, DFS}.  Our overarching goal is to demonstrate
that the right analogues for subschemes in a smooth complete toric variety use
virtual resolutions rather than minimal free resolutions.  This distinction is
not apparent on projective space because the New Intersection
Theorem~\cites{Roberts} establishes that a free complex with finite-length
higher homology groups has to be at least as long as the minimal free
resolution.  For other toric varieties such as products of projective spaces,
allowing irrelevant homology may yield simpler complexes; see
Example~\ref{exa:curveI}.

Throughout this paper, we write
$\PP^{\bbn} \coloneq \PP^{n_1} \times \PP^{n_2} \times \dotsb \times
\PP^{n_r}$ for the product of projective spaces with dimension vector
$\bbn \coloneq (n_1, n_2, \dotsc, n_r) \in \NN^r$ over a field $\kk$.  Let
$S \coloneq \kk[x_{i,j} : \text{$1 \leq i \leq r$, $0 \leq j \leq n_i$}]$ be
the Cox ring of $\PP^{\bbn}$ and let
$B \coloneq \bigcap_{i=1}^r \ideal{x_{i,0}, x_{i,1}, \dotsc, x_{i,n_i}}$ be
its irrelevant ideal.  We identify the Picard group of $\PP^{\bbn}$ with
$\ZZ^r$ and partially order the elements via their components.  If
$\bbe_1, \bbe_2, \dotsc, \bbe_r$ is the standard basis of $\ZZ^r$, then the
polynomial ring $S$ has the $\ZZ^r$-grading induced by
$\deg(x_{i,j}) \coloneq \bbe_i$.  We first reprove the existence of short
virtual resolutions; compare with Corollary~2.14 in
\cite{eisenbud-erman-schreyer-tate-products}.

\begin{proposition}
  \label{pro:splendidComplexesExist}
  Every finitely-generated $\ZZ^r$-graded $B$-saturated $S$-module has a
  virtual resolution of length at most
  $\abs{\bbn} \coloneq n_1 + n_2 + \dotsb + n_r = \dim \PP^{\bbn}$.
\end{proposition}

\noindent
Since $\dim S - \dim \PP^{\bbn} = r$, we see that a minimal free resolution
can be arbitrarily long when compared with a virtual resolution.  A proof of
Proposition~\ref{pro:splendidComplexesExist}, which relies on a locally-free
resolution of the structure sheaf for the diagonal embedding
$\PP^{\bbn} \hookrightarrow \PP^{\bbn} \times \PP^{\bbn}$, appears in
Section~\ref{sec:hst}.

Besides having shorter representatives, virtual resolutions also exhibit a
closer relationship with Castelnuovo--Mumford regularity than minimal free
resolutions.  On projective space, Castelnuovo--Mumford regularity has two
equivalent descriptions: one arising from the vanishing of sheaf cohomology
and another arising from the Betti numbers in a minimal free resolutions.
However, on more general toric varieties, %such as $\PP^{\bbn}$
the multigraded Castelnuovo--Mumford regularity is not determined by a minimal
free resolution; see Theorem~1.5 in \cite{maclagan-smith} or Theorem~4.7 in
\cite{BC}.  From this perspective, we demonstrate that virtual resolutions
improve on minimal free resolutions in two ways.  First,
Theorem~\ref{thm:linearResolutions} proves that the set of virtual resolutions
of a module determines its multigraded Castelnuovo--Mumford regularity.
Second, the next theorem, from Section~\ref{sec:winnow}, demonstrates how to
use regularity to extract a virtual resolution from a minimal free resolution.

\begin{theorem}
  \label{thm:winnow}
  Let $M$ be a finitely-generated $\ZZ^r$-graded $B$-saturated $S$-module that
  is $\bbd$-regular. If $G$ is the free subcomplex of a minimal free
  resolution of $M$ consisting of all summands generated in degree at most
  $\bbd+\bbn$, then $G$ is a virtual resolution of $M$.
\end{theorem}

\noindent
This subcomplex is seldom a resolution.  For convenience, we refer to the free
complex $G$ as the \define{virtual resolution of the pair $(M,\bbd)$}.
Algorithm~\ref{alg:computeWinnow} shows that it can be constructed without
computing the entire minimal free resolution.

The following example illustrates that a virtual resolution can be much
shorter and much thinner than the minimal free resolution.  It follows that a
majority of the summands in the minimal free resolution are unneeded when
building a locally-free resolution of the structure sheaf.

\begin{example}
  \label{exa:curveI}
  A hyperelliptic curve $C$ of genus $4$ can be embedded as a curve of
  bidegree $(2,8)$ in $\PP^1 \times \PP^2$; see Theorem~IV.5.4 in
  \cite{Hartshorne}.  For instance, the $B$-saturated $S$-ideal
  \[
    I \coloneq \left\langle
      \begin{array}{>{\raggedright\tiny}p{415pt}}
        $x_{1,1}^3 x_{2,0}^{} - x_{1,1}^3 x_{2,1}^{} + x_{1,0}^3 x_{2,2}^{}$, 
        $x_{1,0}^2 x_{2,0}^2 + x_{1,1}^2 x_{2,1}^2 + x_{1,0}^{} x_{1,1}^{} x_{2,2}^2$,
        $x_{1,1}^2 x_{2,0}^3 - x_{1,1}^2 x_{2,0}^2 x_{2,1}^{} - x_{1,0}^{}
        x_{1,1}^{} x_{2,1}^2 x_{2,2}^{} - x_{1,0}^2 x_{2,2}^3$, $x_{1,0}^{}
        x_{1,1}^{} x_{2,0}^3 +$ $\relphantom{+} x_{1,0}^{} x_{1,1}^{} x_{2,0}^2
        x_{2,1}^{} - x_{1,0}^2 x_{2,1}^2 x_{2,2}^{} + x_{1,1}^2 x_{2,0}^{}
        x_{2,2}^2 + x_{1,1}^2 x_{2,1}^{} x_{2,2}^2$, $x_{1,1}^{} x_{2,0}^3
        x_{2,1}^2 + x_{1,1}^{} x_{2,0}^2 x_{2,1}^3 - x_{1,0}^{} x_{2,1}^4
        x_{2,2}^{} - x_{1,0}^{} x_{2,0}^3 x_{2,2}^2 + x_{1,0}^{}
        x_{2,0}^2 x_{2,1}^{} x_{2,2}^2 -$ $\relphantom{-} x_{1,1}^{} x_{2,0}^{}
        x_{2,2}^4 - x_{1,1}^{} x_{2,1}^{} x_{2,2}^4$, $x_{1,1}^{} x_{2,0}^5 +
        x_{1,1}^{} x_{2,0}^4 x_{2,1}^{} - x_{1,0}^{} x_{2,0}^2 x_{2,1}^2
        x_{2,2}^{} + x_{1,1}^{} x_{2,1}^2 x_{2,2}^3 + x_{1,0}^{} x_{2,2}^5$, 
        $x_{1,0}^{} x_{2,0}^5 + x_{1,0}^{} x_{2,0}^4 x_{2,1}^{} + x_{1,1}^{} x_{2,1}^4
        x_{2,2}^{} +$ $\relphantom{+} x_{1,1}^{} x_{2,0}^3 x_{2,2}^2 + x_{1,1}^{}
        x_{2,0}^2 x_{2,1}^{} x_{2,2}^2 + x_{1,0}^{} x_{2,1}^2 x_{2,2}^3$,
        $x_{2,0}^8 + 2 x_{2,0}^7 x_{2,1}^{} + x_{2,0}^6 x_{2,1}^2 + x_{2,1}^6
        x_{2,2}^2 + 3 x_{2,0}^3 x_{2,1}^2 x_{2,2}^3 + 3 x_{2,0}^2 x_{2,1}^3
        x_{2,2}^3 - x_{2,0}^{} x_{2,2}^7 - x_{2,1}^{} x_{2,2}^7$
      \end{array}
    \right\rangle 
  \]
  defines such a curve. \emph{Macaulay2}~\cite{M2} shows that the minimal free
  resolution of $S/I$ has the form
  \begin{align}
    \label{eqn:resolution curve}
    &S^1 \gets 
      \begin{matrix} 
        S(-3,-1)^1 \\[-3pt]
        \oplus     \\[-3pt]
        S(-2,-2)^1 \\[-3pt]
        \oplus     \\[-3pt]
        S(-2,-3)^2 \\[-3pt]
        \oplus     \\[-3pt]
        S(-1,-5)^3 \\[-3pt]
        \oplus     \\[-3pt]
        S(0,-8)^1
      \end{matrix} 
    \xleftarrow{\;\; \relphantom{\varphi} \;\;}
      \begin{matrix} 
        S(-3,-3)^3 \\[-3pt]
        \oplus     \\[-3pt] 
        S(-2,-5)^6 \\[-3pt] 
        \oplus     \\[-3pt] 
        S(-1,-7)^1 \\[-3pt] 
        \oplus     \\[-3pt]
        S(-1,-8)^2
      \end{matrix}
    \gets 
      \begin{matrix} 
        S(-3,-5)^3 \\[-3pt] 
        \oplus     \\[-3pt] 
        S(-2,-7)^2 \\[-3pt] 
        \oplus     \\[-3pt] 
        S(-2,-8)^1
      \end{matrix}
    \gets S(-3,-7)^1 \gets 0 \, .
    \intertext{Using the Riemann--Roch
    Theorem~\cite{Hartshorne}*{Theorem~IV.1.3}, one verifies that the module
    $S/I$ is $(4,2)$-regular, so the virtual resolution of the pair $\bigl(S/I,
    (4,2) \bigr)$ has the much simpler
    form}
    \label{eqn:aCMcurve}
    & S^1 \gets 
      \begin{matrix} 
        S(-3,-1)^1 \\[-3pt] 
        \oplus \\[-3pt] 
        S(-2,-2)^1 \\[-3pt] 
        \oplus \\[-3pt] 
        S(-2,-3)^2
      \end{matrix} 
    \xleftarrow{\;\; \varphi \;\;} S(-3,-3)^3 \gets 0 \, .
  \end{align}
  If the ideal $J \subset S$ is the image of the first map in
  \eqref{eqn:aCMcurve}, then we have $J = I \cap Q$ for some ideal $Q$ whose
  radical contains the irrelevant ideal.  Using
  Proposition~\ref{pro:resolutionBound}, we can even conclude that $S/J$ is
  Cohen--Macaulay and $J$ is the ideal of maximal minors of the $4 \times 3$
  matrix
  \begin{equation}
    \label{eqn:phi}
    \varphi \coloneq
    \begin{bmatrix}
      x_{2,1}^{2}                          & x_{2,2}^{2} & -x_{2,0}^{2}     \\
      - x_{1,1}^{} (x_{2,0}^{} - x_{2,1}^{}) & 0      & x_{1,0}^{} x_{2,2}^{} \\
      x_{1,0}^{}                           & -x_{1,1}^{} & 0               \\
      0                                   & x_{1,0}^{}  & x_{1,1}^{}       \\
    \end{bmatrix} \, .
    \qedhere
  \end{equation}
\end{example} 

As an initial step towards our larger goal, we formulate a novel analogue for
properties of points in projective space.  Although any punctual subscheme of
projective space is arithmetically Cohen--Macaulay, this almost always fails
for a zero-dimensional subscheme of $\PP^{\bbn}$; see \cite{GV}.  However, we
do obtain a short virtual resolution just by choosing an unconventional module
to represent the structure sheaf on the punctual subscheme.
 
\begin{theorem}
  \label{thm:ptsACM}
  Let $Z \subset \PP^{\bbn}$ be a zero-dimensional scheme and let $I$ be its
  corresponding $B$-saturated $S$-ideal.  There exists an $S$-ideal $Q$, whose
  radical contains $B$, such that the minimal free resolution of
  $S/(I \cap Q)$ has length $\abs{\bbn}$.  In particular, the minimal free
  resolution of $S/(I \cap Q)$ is a virtual resolution of $S/I$.
\end{theorem}

\noindent
This theorem, proven in Section~\ref{sec:punctual}, does not imply that
$S / (I \cap Q)$ is itself Cohen--Macaulay, as the components of $Q$ will
often have codimension less than $\abs{\bbn}$.  However, when the ambient
variety is $\PP^1 \times\PP^1$, the ring $S / (I \cap Q)$ will be
Cohen--Macaulay of codimension $2$.  In this case,
Corollary~\ref{cor:HilbertBurch} shows that there is a matrix whose maximal
minors cut out $Z$ scheme-theoretically.
Proposition~\ref{pro:pointsInGeneralPosition} extends this to general points
on any smooth toric surface.

As a second and perhaps more substantial step, we apply virtual resolutions to
deformation theory.  On projective space, there are three classic situations
in which the particular structure of the minimal free resolution allows one to
show that all deformations have the same structure: arithmetically
Cohen--Macaulay subschemes of codimension $2$, arithmetically Gorenstein
subschemes in codimension $3$, and complete intersections; see
Sections~2.8--2.9 in \cite{hartshorne-deformation}.  We generalize these
results about unobstructed deformations in projective space as follows.

\begin{theorem}
  \label{thm:unobstructedDefs}
  Consider $Y \subset \PP^{\bbn}$ and let $I$ be the corresponding
  $B$-saturated $S$-ideal.  Assume that the generators of $I$ have degrees
  $\bbd_1,\bbd_2, \dots,\bbd_s$ and that the natural map
  $(S/I)_{\bbd_i}\to H^0\bigl( Y, \cO_{Y}^{}(\bbd_i) \bigr)$ is an isomorphism
  for all $1 \leq i \leq s$.  If any one of the following conditions hold
  \begin{enumerate}[\upshape (i)]
  \item the subscheme $Y$ has codimension $2$ and there is
    $\bbd\in \reg (S/I)$ such that the virtual resolution of the pair
    $(S/I,\bbd)$ has length $2$;
  \item each factor in $\PP^{\bbn}$ has dimension at least $2$, the subscheme
    $Y$ has codimension $3$, and there is $\bbd \in \reg (S/I)$ such that the
    virtual resolution of the pair $(S/I,\bbd)$ is a self-dual complex (up to a
    twist) of length $3$; or
  \item there is $\bbd \in \reg (S/I)$ such that virtual resolution of the pair
    $(S/I,\bbd)$ is a Koszul complex of length $\codim Y$;
  \end{enumerate}
  then the embedded deformations of $Y$ in $\PP^{\bbn}$ are unobstructed and
  the component of the multigraded Hilbert scheme of $\PP^{\bbn}$ containing
  the point corresponding to $Y$ is unirational.
\end{theorem}

To illustrate this theorem, we can reuse the hyperelliptic curve in
Example~\ref{exa:curveI}.

\begin{example}
  \label{exa:curveII}
  By reinterpreting Example~\ref{exa:curveI}, we see that the hyperelliptic
  curve $C \subset \PP^1 \times \PP^2$ satisfies condition~(i) in
  Theorem~\ref{thm:unobstructedDefs}.  It follows that the embedded
  deformations of $C$ are unobstructed and the corresponding component of the
  multigraded Hilbert scheme of $\PP^1 \times \PP^2$ can be given an explicit
  unirational parametrization by varying the entries in the $4 \times 3$
  matrix $\varphi$ from \eqref{eqn:phi}.
\end{example}

Three other geometric applications for virtual resolutions are collected in
Section~\ref{sec:applications}.  The first, Proposition~\ref{pro:unmixedness},
provides an unmixedness result for subschemes of $\PP^{\bbn}$ that have a
virtual resolution whose length equals its codimension. The second,
Proposition~\ref{pro:regularityTensor}, gives sharp bounds on the
Castelnuovo--Mumford regularity of a tensor product of coherent
$\cO_{\PP^{\bbn}}^{}$-modules; compare with Proposition~1.8.8 in
\cite{lazarsfeld}.  Lastly, Proposition~\ref{pro:pushforward} describes new
vanishing results for the higher-direct images of sheaves, which are optimal
in many cases.

The final section presents some promising directions for future research.

%% ---------------------------------------------------------------------------
\subsection*{Conventions}

In this article, we work in the product
$\PP^{\bbn} \coloneq \PP^{n_1} \times \PP^{n_2} \times \dotsb \times
\PP^{n_r}$ of projective spaces with dimension vector
$\bbn \coloneq (n_1, n_2, \dotsc, n_r) \in \NN^r$ over a field $\kk$.  Its Cox
ring is the polynomial ring $S \coloneq \kk[x_{i,j} : \text{$1 \leq i \leq r$,
  $0 \leq j \leq n_i$}]$ and its irrelevant ideal is
$B \coloneq \bigcap_{i=1}^r \ideal{x_{i,0}, x_{i,1}, \dotsc, x_{i,n_i}}$.  The
Picard group of $\PP^{\bbn}$ is identified with $\ZZ^r$ and the elements are
partially ordered componentwise.  If $\bbe_1, \bbe_2, \dotsc, \bbe_r$ is the
standard basis of $\ZZ^r$, then $S$ has the $\ZZ^r$-grading induced by
$\deg(x_{i,j}) \coloneq \bbe_i$.  We assume that all $S$-modules are finitely
generated and $\ZZ^r$-graded.

%% ---------------------------------------------------------------------------
\subsection*{Acknowledgements}

Some of this research was completed during visits to the Banff International
Research Station (BIRS) and the Mathematical Sciences Research Institute
(MSRI), and we are very grateful for their hospitality.  We thank Lawrence
Ein, David Eisenbud, Craig Huneke, Nathan Ilten, Rob Lazarsfeld, Mike Loper,
Diane Maclagan, Frank-Olaf Schreyer, and Ian Shipman for helpful
conversations.  We also thank an anonymous referee for their valuable
suggestions.

%%%%%%%%%%%%%%%%%%%%%%%%%%%%%%%%%%%%%%%%%%%%%%%%%%%%%%%%%%%%%%%%%%%%%%%%%%%%%%
\section{Existence of Short Virtual Resolutions}
\label{sec:hst}

\noindent 
This section, by proving Proposition~\ref{pro:splendidComplexesExist},
establishes the existence of virtual resolutions whose length is bounded above
by the dimension of $\PP^{\bbn}$.  In particular, these virtual resolutions
are typically shorter than a minimal free resolution.  Moreover,
Proposition~\ref{pro:resolutionBound} shows that
Proposition~\ref{pro:splendidComplexesExist} provides the best possible
uniform bound on the length of a virtual resolution.  Exploiting multigraded
Castelnuovo--Mumford regularity, we also produce short virtual resolutions
where the degrees of the generators of the free modules satisfy explicit
bounds.  Better yet, we obtain a converse, by showing that the set of virtual
resolutions of a module determine its regularity.

Our proof of Proposition~\ref{pro:splendidComplexesExist} is based on a minor
variation of Beilinson's resolution of the diagonal; compare with
Proposition~3.2 in \cite{caldararu} or Lemma~8.27 in \cite{huybrechts}.  Given
an $\cO_{\!X_{\!j}}^{}$-module $\cF_{\!j}$ for all $1 \leq j \leq n$, their
external tensor product is
\[
  \cF_1 \boxtimes \cF_2 \boxtimes \dotsb \boxtimes \cF_m \coloneq (\pr_1^*
  \cF_1) \otimes_{\cO_{X}} (\pr_2^*\cF_2) \otimes_{\cO_{X}} \dotsb
  \otimes_{\cO_{X}} (\pr_m^* \cF_m)
\]
where $\pr_j$ denotes the projection map from the Cartesian product
$X \coloneq X_1 \times X_2 \times \dotsb \times X_m$ to $X_j$.  In particular,
for all $\bbu \in \ZZ^r$, we have
$\cO_{\PP^{\bbn}}^{}(\bbu) = \cO_{\PP^{n_1}}^{}(u_1) \boxtimes
\cO_{\PP^{n_2}}^{}(u_2) \boxtimes \dotsb \boxtimes \cO_{\PP^{n_r}}^{}(u_r)$.
With this notation, we can describe the resolution of the diagonal
$\PP^{\bbn} \hookrightarrow \PP^{\bbn} \times \PP^{\bbn}$.

\begin{lemma}
  \label{lem:Diagonal}
  If
  $\cT_{\PP^{\bbn}}^{\bbe_i} \coloneq \cO_{\PP^{n_1}}^{} \boxtimes
  \cO_{\PP^{n_2}}^{} \boxtimes \dotsb \boxtimes \cO_{\PP^{n_{i-1}}}^{}
  \boxtimes \cT_{\PP^{n_i}}^{} \boxtimes \cO_{\PP^{n_{i+1}}}^{} \boxtimes
  \dotsb \boxtimes \cO_{\PP^{n_r}}^{}$ for $1 \leq i \leq r$, then the
  diagonal $\PP^{\bbn} \hookrightarrow \PP^{\bbn} \times \PP^{\bbn}$ is the
  zero scheme of a global section of
  $\bigoplus_{i=1}^r \cO_{\PP^{\bbn}}^{}(\bbe_i) \boxtimes
  \cT_{\PP^{\bbn}}^{\bbe_i}(-\bbe_i)$.  Hence, the diagonal has a locally-free
  resolution of the form
  \begin{equation*}
    \cO_{\PP^{\bbn} \times \PP^{\bbn}}^{}
    \gets \bigoplus_{i=1}^r \!\cO_{\PP^{\bbn}}^{}(-\bbe_i) \boxtimes
    \Omega_{\PP^{\bbn}}^{\bbe_i} (\bbe_i) \gets \!\!\!\!\!\!
    \bigoplus_{\substack{\boldzero \leq \bbu \leq \bbn\\ \abs{\bbu} = 2}}
    \!\!\!\!\!
    \cO_{\PP^{\bbn}}^{}(-\bbu) \boxtimes \Omega_{\PP^{\bbn}}^{\bbu} (\bbu)
    \leftarrow \!\dotsb\! \leftarrow \cO_{\PP^{\bbn}}^{}(-\bbn) \boxtimes
    \Omega_{\PP^{\bbn}}^{\bbn}(\bbn) ,
  \end{equation*}
  where
  $\Omega_{\PP^{\bbn}}^{\bba} \coloneq \Omega_{\PP^{n_1}}^{a_1} \boxtimes
  \Omega_{\PP^{n_2}}^{a_2} \boxtimes \dotsb \boxtimes
  \Omega_{\PP^{n_r}}^{a_r}$
  is the external tensor product of the exterior powers of the cotangent
  bundles on the factors of $\PP^{\bbn}$.
\end{lemma}

\begin{proof}
  For each $1 \leq i \leq r$, fix a basis
  $x_{i,0}^{}, x_{i,1}^{}, \dotsc, x_{i,n_i}^{}$ for
  $H^0\bigl( \PP^{\bbn}, \cO_{\PP^{\bbn}}^{}(\bbe_i) \bigr)$.  The Euler sequence
  on $\PP^{n_i}$ yields
  \[
    0 \longleftarrow \cT_{\PP^{\bbn}}^{\bbe_i} \longleftarrow \bigoplus_{j =
      0}^{n_i} \cO_{\PP^{\bbn}}^{}(\bbe_i) \xleftarrow{\;\;
      \left[ 
        \begin{smallmatrix} 
          x_{i,0} & x_{i,1} & \dotsb & x_{i,n_i}
        \end{smallmatrix} 
      \right] \;\;} \cO_{\PP^{\bbn}}^{} \longleftarrow 0 \, ;
  \]
  see Theorem~8.1.6 in \cite{CLS}.  Knowing the cohomology of line bundles on
  $\PP^{\bbn}$, the associated long exact sequence gives
  $H^0 \bigl( \PP^{\bbn}, \cT_{\PP^{\bbn}}^{\bbe_i}(-\bbe_i) \bigr) \cong
  \bigoplus_{j=0}^{n_i} H^0(\PP^{\bbn}, \cO_{\PP^{\bbn}}^{})$.  A basis for
  $\bigoplus_{j=0}^{n_i} H^0(\PP^{\bbn}, \cO_{\PP^{\bbn}}^{})$ is given by the
  dual basis $x_{i,0}^\ast, x_{i,1}^\ast, \dotsc, x_{i,n_i}^\ast$.  Let
  $\partial/\partial x_{i,j}$ denote the image of $x_{i,j}^*$ in
  $H^0 \bigl( \PP^{\bbn}, \cT_{\PP^{\bbn}}(-\bbe_i) \bigr)$.

  Consider
  $s \in H^0\bigl( \PP^{\bbn} \times \PP^{\bbn}, \bigoplus_{i=1}^r
  \cO_{\PP^{\bbn}}^{}(\bbe_i) \boxtimes \cT_{\PP^{\bbn}}^{\bbe_i}(-\bbe_i)
  \bigr)$ given by
  \[
  s \coloneq \left( \sum_{j = 0}^{n_1} x_{1,j} \frac{\partial}{\partial y_{1,j}},
    \sum_{j = 0}^{n_2} x_{2,j} \frac{\partial}{\partial y_{2,j}}, \dotsc,
    \sum_{j = 0}^{n_r} x_{r,j} \frac{\partial}{\partial y_{r,j}} \right)
  \]
  where $x_{i,j}$ and $y_{i,j}$ are the coordinates on the first and second
  factor of $\PP^{\bbn} \times \PP^{\bbn}$ respectively.  We claim that the
  zero scheme of $s$ equals the diagonal in $\PP^{\bbn} \times \PP^{\bbn}$.
  By symmetry, it suffices to check this on a single affine open
  neighborhood.  If $x_{1,0} x_{2,0} \dotsb x_{r,0} \neq 0$ and
  $y_{1,0} y_{2,0} \dotsb y_{r,0} \neq 0$, then the Euler relations yield
  \begin{align*}
    \sum_{j = 0}^{n_i} x_{i,j} \frac{\partial}{\partial y_{i,j}} 
    &= x_{i,0} \left( - \frac{1}{y_{i,0}} \sum_{j =
      1}^{n_i} y_{i,j} \frac{\partial}{\partial y_{i,j}} \right) + \sum_{j =
      1}^{n_i} x_{i,j} \frac{\partial}{\partial y_{i,j}} = \frac{1}{y_{i,0}}
      \sum_{j = 1}^{n_i} (x_{i,j}y_{i,0} - x_{i,0} y_{i,j})
      \frac{\partial}{\partial y_{i,j}} \, ,
  \end{align*}
  for each $1 \leq i \leq r$.  It follows that $s = 0$ if and only if
  $\frac{x_{i,j}}{x_{i,0}} = \frac{y_{i,j}}{y_{i,0}}$ for all
  $1 \leq i \leq r$ and $1 \leq j \leq n_i$.  Hence, the global section $s$
  vanishes precisely on the diagonal
  $\PP^{\bbn} \hookrightarrow \PP^{\bbn} \times \PP^{\bbn}$.

  The Koszul complex associated to $s$ is the required locally-free resolution
  of the diagonal, because $\PP^{\bbn}$ is smooth and the codimension of the
  diagonal equals the rank of the vector bundle
  $\bigoplus_{i=1}^r \cO_{\PP^{\bbn}}^{}(\bbe_i) \boxtimes
  \cT_{\PP^{\bbn}}^{\bbe_i}(-\bbe_i)$; see Section~B.2 in \cite{lazarsfeld}.
  Since
  $\Omega_{\PP^{\bbn}}^{\bbe_i} = \cHom_{\cO_{\PP^{\bbn}}}(
  \cT_{\PP^{\bbn}}^{\bbe_i}, \cO_{\PP^{\bbn}}^{})$, we have
  \[
  \bigwedge^k \bigl( \cO_{\PP^{\bbn}}^{}(-\bbe_i) \boxtimes
  \Omega_{\PP^{\bbn}}^{\bbe_i}(\bbe_i) \bigr) = \bigoplus_{\substack{\boldzero
      \leq \bbu \leq \bbn\\ \abs{\bbu} = k}} \cO_{\PP^{\bbn}}^{}(-\bbu) \boxtimes
  \Omega_{\PP^{\bbn}}^{\bbu} (\bbu) \, ,
  \]
  for $0 \leq k \leq \abs{\bbn}$.
\end{proof}

\begin{proof}[Proof of Proposition~\ref{pro:splendidComplexesExist}]
  Let $\pi_1$ and $\pi_2$ be the projections of $\PP^{\bbn}\times \PP^{\bbn}$
  onto the first and second factors respectively.  For any $\bbu \in \ZZ^r$,
  the Fujita Vanishing Theorem~\cite{fujita-semipositive}*{Theorem~1} implies
  that $\Omega_{\PP^{\bbn}}^{\bbu}(\bbu + \bbd)\otimes \widetilde{M}$ has no
  higher cohomology for any sufficiently positive $\bbd\in\ZZ^r$.  Let $\cK$
  be the locally-free resolution of the diagonal
  $\PP^{\bbn} \hookrightarrow \PP^{\bbn} \times \PP^{\bbn}$ described in
  Lemma~\ref{lem:Diagonal}.  Both hypercohomology spectral sequences, namely
  \begin{xalignat*}{3}
    {^\prime \mkern-1.2mu\operatorname{E}}_2^{p,q} &\coloneq \HH^p \bigl(
    \mathbf{R}^{q} {\pi_1}_{*} \bigl( \pi_2^* \widetilde{M}(\bbd)
    \otimes_{\cO_{\PP^{\bbn} \times \PP^{\bbn}}} \cK \bigr) \bigr) & &
    \text{and} & {^{\prime\prime}\mkern-1.2mu\operatorname{E}}_2^{p,q}
    &\coloneq \mathbf{R}^{p} {\pi_1}_{*} \, \HH^q \bigl( \pi_2^*
    \widetilde{M}(\bbd) \otimes_{\cO_{\PP^{\bbn} \times \PP^{\bbn}}} \cK
    \bigr) \, ,
  \end{xalignat*}
  converge to
  $\mathbf{R}^{p+q} {\pi_1}_{*}^{} \bigl( \pi_2^* \widetilde{M}(\bbd)
  \otimes_{\cO_{\PP^{\bbn} \times \PP^{\bbn}}} \cK \bigr)$; see Section~12.4
  in \cite{EGA3.1}.  Since $\cK$ is a locally-free resolution of the diagonal,
  it follows that
  ${^{\prime\prime}\mkern-1.2mu\textrm{E}}_2^{0,0} \cong \widetilde{M}(\bbd)$
  and ${^{\prime\prime}\mkern-1.2mu\textrm{E}}_2^{p,q} = 0$ when either
  $p \neq 0$ or $q \neq 0$; compare with Proposition~8.28 in
  \cite{huybrechts}.  Hence, we conclude that
  \[ 
  \mathbf{R}^{p+q} {\pi_1}_{*}\bigl( \pi_2^* \widetilde{M}(\bbd)
  \otimes_{\cO_{\PP^{\bbn} \times \PP^{\bbn}}} \cK \bigr) \cong
  \begin{cases}
    \widetilde{M}(\bbd) & \text{if $p = 0 = q$,} \\
    0 & \text{otherwise.}
  \end{cases}
  \]
  On the other hand, the first page of the other hypercohomology spectral
  sequence is
  \begin{align*}
    {^\prime \mkern-1.2mu\textrm{E}}_1^{p,q} = \mathbf{R}^{q} {\pi_1}_{*} \bigl(
    \pi_2^* \widetilde{M}(\bbd) \otimes_{\cO_{\PP^{\bbn} \times \PP^{\bbn}}}
    \cK_{-p} \bigr) 
    &= \mathbf{R}^{q} {\pi_1}_{*} \Biggl(
      \bigoplus_{\substack{\boldzero \leq \bbu \leq \bbn\\ \abs{\bbu} = -p}}
    \cO_{\PP^{\bbn}}^{}(-\bbu) \boxtimes \big( \Omega_{\PP^{\bbn}}^{\bbu} \otimes
    \widetilde{M}( \bbu + \bbd) \bigr) \Biggr) \\
    &= \bigoplus_{\substack{\boldzero \leq \bbu \leq \bbn\\ \abs{\bbu} = -p}}
    \cO_{\PP^{\bbn}}^{}(-\bbu) \otimes_{\kk} H^q \big( \PP^{\bbn},
    \Omega_{\PP^{\bbn}}^{\bbu} \otimes \widetilde{M}( \bbu + \bbd) \bigr) \, .
  \end{align*}
  Our positivity assumption on $\bbd$ implies that
  $ H^q \big( \PP^{\bbn}, \Omega_{\PP^{\bbn}}^{\bbu} \otimes \widetilde{M}(
  \bbu + \bbd) \bigr) = 0$ for all $q > 0$, so
  ${^\prime \mkern-1.2mu\textrm{E}}_1^{p,q}$ is concentrated in a single row.
  Applying the functor
  $\cF \mapsto \bigoplus_{\bbv \in \NN^r} H^0\bigl( \PP^{\bbn}, \cF(\bbv)
  \bigr)$, we obtain a virtual resolution of $M$ in which the $i$-th module is
  \[
  \bigoplus_{\substack{\boldzero \leq \bbu \leq \bbn\\ \abs{\bbu} = i}}
  S(-\bbu) \otimes_{\kk} H^q \big( \PP^{\bbn}, \Omega_{\PP^{\bbn}}^{\bbu}
  \otimes \widetilde{M}( \bbu + \bbd) \bigr) \, . \qedhere 
  \]
\end{proof}

\begin{remark}
  By scrutinizing the linear free resolutions of well-chosen truncated twisted
  modules, Corollary~2.14 in \cite{eisenbud-erman-schreyer-tate-products} also
  establishes the existence of short virtual resolutions on $\PP^{\bbn}$.
  Although the proof of Proposition~\ref{pro:splendidComplexesExist} and
  Proposition~2.7 in \cite{eisenbud-erman-schreyer-tate-products} use somewhat
  different the notions of a ``sufficiently positive'' degree
  $\bbd \in \ZZ^r$, both are quite similar to Castelnuovo--Mumford regularity.
\end{remark}

The next examples demonstrate why we want more than just these short
virtual resolutions arising from the proof of
Proposition~\ref{pro:splendidComplexesExist}.

\begin{example}
  \label{ex:res of diagonal}
  Consider the hyperelliptic curve $C \subset \PP^1 \times \PP^2$ from
  Example~\ref{exa:curveI}.  Using $\bbd \coloneq (2,2)$ in the construction
  from the proof of Proposition~\ref{pro:splendidComplexesExist} yields a
  virtual resolution of the form
  \[
    S(-2,-2)^{17} \gets 
    \begin{matrix} 
      S(-2,-3)^{26} \\[-3pt] 
      \oplus       \\[-3pt]
      S(-3,-2)^{15}
    \end{matrix} 
    \gets 
    \begin{matrix}
      S(-2,-4)^9   \\[-3pt] 
      \oplus       \\[-3pt]
      S(-3,-3)^{22}
    \end{matrix} 
    \gets S(-3,-4)^7\gets 0.
  \]
  Compared to the virtual resolution in \eqref{eqn:aCMcurve}, the length of
  this complex is longer, the rank of the free modules is higher, and the
  degrees of the generators are larger.
\end{example}

\begin{example}
  If $X$ is the union of $m$ distinct points on $\PP^1 \times \PP^1$, then for
  any sufficiently positive $\bbd = (d_1,d_2)$, the construction in the proof of
  Proposition~\ref{pro:splendidComplexesExist} yields a virtual resolution of the
  form
  \[
    S(-d_1,-d_2)^m \gets 
    \begin{matrix} 
      S(-d_1-1,-d_2)^{m} \\[-3pt] 
      \oplus            \\[-3pt]
      S(-d_1,-d_2-1)^{m}
    \end{matrix}
    \gets S(-d_1-1,-d_2-1)^m \gets 0.
  \]
  Unlike the minimal free resolution, this Betti table of this free complex is
  independent of the geometry of the points, so even short virtual resolutions
  can obscure the geometric information.
\end{example}

As a counterpoint to Proposition~\ref{pro:splendidComplexesExist}, we provide
a lower bound on the length of a virtual resolution.  Extending the well-known
result for projective space, we show that the codimension of any associated
prime of $M$ gives a lower bound on the length of any virtual resolution of
$M$.

\begin{proposition}
  \label{pro:resolutionBound}
  Let $M$ be a finitely-generated $\ZZ^r$-graded $S$-module. Let $Q$ be an
  associated prime of $M$ that does not contain the irrelevant ideal $B$ and
  let $F \coloneq [F_0 \gets F_1 \gets \dotsb \gets F_p \gets 0]$ be a virtual
  resolution of $M$.  These hypotheses yield the following.
  \begin{enumerate}[\upshape(i)]
  \item We have $ \codim Q \leq p$.  
  \item If $Q$ is the prime ideal for a closed point of $\PP^{\bbn}$, then we
    have $p \geq \abs{\bbn}$.
  \item If $p \leq \min \{n_i + 1 : 1 \leq i \leq r \}$, then $F$ is a free
    resolution of $\HH_0(F)$.
  \end{enumerate}
\end{proposition}

\begin{proof}
  Localizing at the prime ideal $Q$, $F$ becomes a free $S_Q$-resolution of
  $M_{Q}$.  Part (i) then follows from the fact that, over the local ring
  $S_Q$, the projective dimension of a module is always greater than or equal
  to the codimension of a module; see Proposition~18.2 in
  \cite{eisenbud-book}.  Part (ii) is immediate, as $\codim Q = \abs{\bbn}$ if
  $Q$ is the prime ideal for a closed point of $\PP^{\bbn}$.

  For part~(iii), assume by way of contradiction that $F$ is not a free
  resolution of $\HH_0(F)$.  It follows that $\HH_j(F) \neq 0$ for some
  $j > 0$; choose the maximal such $j$.  Since $F$ is a virtual resolution of
  $M$, the module $\HH_j(F)$ must be supported on the irrelevant ideal $B$.
  Setting $P_i \coloneq \ideal{x_{i,0}, x_{i,1}, \dotsc, x_{i,n_i}}$ to be the
  component of the irrelevant ideal $B$ corresponding to the factor
  $\PP^{n_i}$, there is an index $i$ such that
  $\bigl( \HH_j(F) \bigr)_{P_i} \neq 0$.  Localizing at $P_i$ yields a complex
  $F_{P_i}$ of the form
  \[
    \dotsb \gets (F_{j-1})_{P_i} \gets (F_{j})_{P_i} \gets (F_{j+1})_{P_i}
    \gets \dotsb \gets (F_p)_{P_i} \gets 0
  \] 
  where $\HH_j(F_{P_i})$ is supported on the maximal ideal of $S_{P_i}$.  We
  deduce that $p > p - j \geq \codim P_i = n_i + 1$ from the Peskine--Szpiro
  Acyclicity Lemma; see Lemma~20.11 in \cite{eisenbud-book}.  However, this
  contradicts our assumption that
  $p \leq \min\{ n_i + 1 : 1 \leq i \leq r \}$.  Therefore, we conclude that
  the complex $F$ is a free resolution of $\HH_0(F)$.
\end{proof}

The following simple corollary is useful in applications such as
Theorem~\ref{thm:unobstructedDefs}.

\begin{corollary}
  \label{cor:splendidEqualsMinimal}
  Let $I$ be a $B$-saturated $S$-ideal and let
  $F = [F_0 \gets F_1\gets \dotsb \gets F_p \gets 0]$ be a virtual resolution
  of $S/I$. If $F_0 = S$ and $p < \min \{ n_i + 1 : 1 \leq i \leq r \}$, then
  the complex $F$ is a free resolution of $S/I$.
\end{corollary}

\begin{proof}
  By part~(iii) of Proposition~\ref{pro:resolutionBound}, the complex $F$ is a
  free resolution of $\HH_0(F)$.  The hypothesis that $F_0 = S$ implies that
  $\HH_0(F) = S/J$ for some ideal $J$.  Since $I$ is $B$-saturated and $F$ is
  a virtual resolution of $S/I$, we deduce that $I$ equals the $B$-saturation
  of $J$.  If we had $I \neq J$, then it would follow that $S/J$ has an
  associated prime $Q$ that contains the irrelevant ideal $B$.  However, the
  codimension of $Q$ is at least $\min\{n_i+1 : 1 \leq i \leq r \}$. As $F$ is
  a free resolution of $S/J$, this would yield the contraction
  $p \geq \min\{n_i+1 : 1 \leq i \leq r \}$; see Proposition~18.2 in
  \cite{eisenbud-book}.
\end{proof}

Just like in projective space, one can find subvarieties of codimension $c$
which do not admit a virtual resolution of length $c$.

\begin{example}
  Working in $\PP^2 \times \PP^2$, consider the $B$-saturated $S$-ideal
  $J \coloneq \ideal{x_{1,0}, x_{1,1}} \cap \ideal{x_{2,0}, x_{2,1}}$. The
  minimal free resolution of $S/J$ has the form
  \[
    S \gets S(-1,-1)^4 \gets 
    \begin{matrix} 
      S(-2, -1)^2 \\[-3pt] 
      \oplus      \\[-3pt]
      S(-1,-2)^2 
    \end{matrix} 
    \gets S(-2,-2) \gets 0.
  \]
  Although the codimension of every associated prime of $J$ is $2$, there is
  no virtual resolution of $S/J$ of length $2$.  If we had such a free complex
  $F = [ F_0 \gets F_1 \gets F_2 \gets 0]$, then
  Corollary~\ref{cor:splendidEqualsMinimal} would imply that $F$ is a minimal
  free resolution of $S/J$, which would be a contradiction.
\end{example}

\begin{remark}
  Proposition~\ref{pro:unmixedness} analyzes when a subscheme has a virtual
  resolution of its structure sheaf whose length equals its codimension---a
  special case of equality in part~(i) of
  Proposition~\ref{pro:resolutionBound}.
\end{remark}

We next refine our results on short virtual resolutions by developing
effective degree bounds.  Following Definition~1.1 in \cite{maclagan-smith}, a
finitely-generated $\ZZ^r$-graded $B$-saturated $S$-module $M$ is
$\bbm$-regular, for some $\bbm \in \ZZ^r$, if $H_B^i(M)_{\bbp} =0$ for all
$i\geq 1$ and all $\bbp \in \bigcup (\bbm - \bbq + \NN^r)$, where the union is
over all $\bbq \in \NN^r$ such that $\abs{\bbq} = i - 1$.  The
\define{(multigraded Castelnuovo--Mumford) regularity of $M$} is
$\reg M \coloneq \{\bbp \in \ZZ^r : \text{$M$ is $\bbp$-regular}\}$.  Let
$\Delta_i \subset \ZZ^r$ denote the set of twists of the summands in the
$i$-th step of the minimal free resolution of the irrelevant ideal $B$.

\begin{theorem}
  \label{thm:linearResolutions}
  Let $M$ be a finitely-generated $\ZZ^r$-graded $B$-saturated $S$-module.  We
  have $\bbd \in \reg M$ if and only if the module $M(\bbd)$ has a virtual
  resolution $F_0 \gets F_1 \gets \dotsb \gets F_{\abs{\bbn}} \gets 0$ such
  that, for all $0 \leq i \leq \abs{\bbn}$, the degree of each generator of
  $F_i$ belongs to $\Delta_i+\NN^r$, and its Hilbert polynomial and Hilbert
  function agree on $\NN^r$.
\end{theorem}

When $r=1$, we have $\Delta_i = \{ -i \}$ and this theorem specializes to the
existence of linear resolutions on projective space; see Proposition~1.8.8 in
\cite{lazarsfeld}.  Since the minimal free resolution of $S/B$ is a cellular
resolution described explicitly by Corollary~2.13 in \cite{bayer-sturmfels},
it follows that $\Delta_0 \coloneq \{ \boldzero \}$ and that for $i\geq 1$, we
have $\Delta_i \coloneq \left\{ - \bba\in \ZZ^r :
  \text{$\boldzero \leq \bba -\boldone \leq \bbn$ and
    $\abs{\bba} = r + i - 1$} \right\}$.  We first illustrate
Theorem~\ref{thm:linearResolutions} in the case of a hypersurface.

\begin{example}
  Given a homogeneous polynomial $f \in S$ of degree $\bbd$, the regularity of
  $S/\ideal{f}$ has a unique minimal element $\bbe$, where
  $ e_j \coloneq \max\{ 0, d_j - 1 \}$.  As a consequence, it follows that
  $\boldzero \in \reg(S/\ideal{f})$ if and only if $d_j \leq 1$ for all $j$.
\end{example}

Before proving Theorem~\ref{thm:linearResolutions}, we need two technical
lemmas.

\begin{lemma}
  \label{lem:vanishing bit}
  For $0 \leq i \leq \abs{\bbn}$, $\bbb \in \Delta_i + \NN^r$, and
  $\bba \in \NN^r\setminus\{\boldzero\}$, we have
  $H^{\abs{\bba} + i} \bigl( \PP^{\bbn}, \cO_{\PP^{\bbn}}^{}(\bbb-\bba) \bigr)
  = 0$.
\end{lemma}

\begin{proof}
  We induct on $r$.  For the base case $r=1$, we have nonzero higher
  cohomology for the given line bundle on $\PP^{n_1}$ only if
  $\abs{\bba} + i = a_1 + i = n_1$.  Since
  $\bbb \in \Delta_i + \NN = \{ -i \} + \NN$, or $b_1\geq -i$, we have
  $b_1 - a_1 \geq - i - (n_1 - i) = -n_1 > - n_1 - 1$, so
  $\cO_{\PP^{\bbn}}^{}(b_1-a_1)$ has no higher cohomology.
  
  For the induction step, we first consider the case where at least one entry
  of $\bbb - \bba$ is nonnegative, and we assume for contradiction that
  $ H^{\abs{\bba} + i}\bigl(\PP^{\bbn},\cO_{\PP^{\bbn}}^{}(\bbb-\bba) \bigr)
  \ne 0.  $ Since $\abs{\bba} + i > 0$, we may also assume, by reordering the
  factors, that the first entry of $\bbb - \bba$ is strictly negative and the
  last entry is nonnegative.  We write $\bbb - \bba = (\bbb'-\bba', b_r-a_r)$
  and $\bbn = (\bbn',n_r)$ in $\ZZ^{r-1}\oplus \ZZ$.  Since
  $H^{\abs{\bba} + i}(\PP^{\bbn}, \cO_{\PP^{\bbn}}^{}(\bbb-\bba)) \ne 0$, the
  K\"unneth formula implies that
  $ H^{\abs{\bba} + i} \bigl( \PP^{\bbn'}, \cO_{\PP^{\bbn'}}^{}(\bbb'-\bba')
  \bigr) \ne 0.$ Decreasing the first entry of $\bbb'-\bba'$ will not alter
  this nonvanishing.  Setting $\bba''\coloneq\bba' + (a_r,0,0,\dots,0)$, we
  obtain $\bba'' \in \NN^{r-1} \setminus \{\boldzero\}$,
  $\abs{\bba''} = \abs{\bba'} + a_r = \abs{\bba}$, and
  $ H^{\abs{\bba''} + i} \bigl( \PP^{\bbn'},
  \cO_{\PP^{\bbn'}}^{}(\bbb'-\bba'') \bigr) \neq 0$, which contradicts the
  induction hypothesis.

  It remains to consider the case in which all entries of $\bbb-\bba$ are
  strictly negative.  Hence, we can assume that $\abs{\bba} + i = \abs{\bbn}$.
  The hypothesis $\bbb\in \Delta_i+\NN^r$ implies $\abs{\bbb} \geq -r - i +1$.
  Combining these yields
  $\abs{\bbb-\bba} = \sum_{i} b_i-a_i = \abs{\bbb} - \abs{\bba} \geq (-r-i+1)
  - ( \abs{\bbn} - i) = - \abs{\bbn} - r + 1$. But the most positive line
  bundle with top-dimensional cohomology is the canonical bundle, and this
  inequality shows that $\cO_{\PP^{\bbn}}^{}(\bbb-\bba)$ cannot have top
  dimensional cohomology.
\end{proof}

\begin{remark}
  Proposition~\ref{pro:pushforward} develops a related vanishing result for
  derived pushforwards.
\end{remark}

\begin{lemma}
  \label{lem:noHigherCohomology}
  Let $\cF$ be a $\boldzero$-regular $\cO_{\PP^{\bbn}}^{}$-module and let
  $\boldzero\leq \bba \leq \bbn$.  If
  $H^p\bigl(\PP^{\bbn},\cF \otimes \Omega^{\bba}(\bba)\bigr)\ne 0$, then we
  have $-\bba \in \Delta_{\abs{\bba} - p}+\NN^r$.
\end{lemma}

\begin{proof}
  If $\bba = \boldzero$, then we have
  $\Omega^{\bba}(\bba) = \cO_{\PP^{\bbn}}^{}$, and the statement follows
  immediately from the $\boldzero$-regularity of $\cF$.  Thus, we assume that
  $\bba \neq \boldzero$.  After possibly reordering the factors of
  $\PP^{\bbn}$, we may write
  $\bba = (\bba',\boldzero)\in \ZZ^{r'} \oplus \ZZ^{r-r'}$ where every entry
  of $\bba'$ is strictly positive.  For any $k \geq 1$, we have
  $(-\bba',\boldzero) \in \Delta_{k}+\NN^r \iff (-\bba', - \boldone) \in
  \Delta_{k}+\NN^r \iff \abs{\bba'}+(r-r') \leq k+r-1$.  Setting
  $k = \abs{\bba} - p = \abs{\bba'} - p$ establishes that
  $-\bba=(-\bba',\boldzero) \in \Delta_{\abs{\bba} - p} + \NN^r$ is equivalent
  to $p < r'$.
  
  We next use truncated Koszul complexes to build a locally free resolution of
  $\Omega^{\bba}_{\PP^{\bbn}}(\bba)$.  For $j>r'$, we have $a_j=0$ and
  $\Omega^{a_j}_{\PP^{n_j}}(a_j) \cong \cO_{\PP^{n_j}}^{}$.  For
  $1\leq j\leq r'$, the truncated Koszul complex twisted by
  $\cO_{\PP^{n_j}}^{}(a_j)$, namely
  \[
    \cO_{\PP^{n_j}}^{}(-1)^{\binom{n_j+1}{a_j+1}} \gets
    \cO_{\PP^{n_j}}^{}(-2)^{\binom{n_j+1}{a_j+2}}\gets \dotsb \gets
    \cO_{\PP^{n_j}}^{}(-n_j-1+a_j)^{\binom{n_j+1}{n_j+1}} \gets 0,
  \]
  resolves $\Omega^{a_j}_{\PP^{n_j}}(a_j)$. Taking external tensor products
  gives a locally-free resolution $\mathcal G$ of
  $\Omega^{\bba}_{\PP^{\bbn}}(\bba)$.  Any summand $\cO_{\PP^{\bbn}}^{}(\bbc)$
  in $\mathcal G_i$ has the form
  $\bbc = (\bbc',\boldzero)\in \ZZ^{r'}\oplus \ZZ^{r-r'}$, where
  $ \abs{\bbc'} = - r' - i$.  Tensoring the locally-free resolution
  $\mathcal G$ with $\cF$ is a resolution of
  $\cF \otimes \Omega^{\bba}_{\PP^{\bbn}}(\bba)$.

  Since $H^p(\cF \otimes \Omega^{\bba}_{\PP^{\bbn}}(\bba)) \neq 0$, breaking
  the resolution $\cF \otimes \mathcal{G}$ into short exact sequences implies
  that, for some index $i$, we have
  $H^{p+i}(\cF \otimes \mathcal{G}_i) \neq 0$.  Hence, there exists
  $\bbc = (\bbc',\boldzero)$ with $\abs{\bbc'} = - r' - i$ such that
  $H^{p+i}\bigl(\PP^{\bbn}, \cF(\bbc) \bigr) \neq 0$.  Since $\cF$ is
  $\boldzero$-regular, we have that $\abs{\bbc} = \abs{\bbc'} < -(p+i)$.
  Therefore, we conclude that $p < - \abs{\bbc'} -i = (r'+ i) - i = r'$ and
  $\bba \in \Delta_{\abs{\bba} - p} + \NN^r$.
\end{proof}

\begin{proof}[Proof of Theorem~\ref{thm:linearResolutions}]
  Assume that $M(\bbd)$ has a virtual resolution $F$ of the specified form and
  its Hilbert polynomial and Hilbert function agree on $\NN^r$.  Since $M$ is
  $B$-saturated, it suffices to show that
  $H^{\abs{\bba}}(\PP^{\bbn}, \widetilde{M}(\bbd-\bba)) = 0$ for all
  $\bba \in \NN^r - \{\boldzero\}$.  By splitting up $F$ into short exact
  sequences, it suffices to show that
  $H^{\abs{\bba} + i}(\PP^{\bbn},\widetilde{F}_i(-\bba))=0$ for all
  $\bba\in \NN^r \setminus\{\boldzero\}$.  This is the content of
  Lemma~\ref{lem:vanishing bit}.
  
  For the converse, let $\cK$ denote the locally-free resolution of the
  diagonal $\PP^{\bbn} \hookrightarrow \PP^{\bbn} \times \PP^{\bbn}$ described
  in Lemma~\ref{lem:Diagonal}.  Let $\pi_1$ and $\pi_2$ be the projections
  onto the first and second factors of $\PP^{\bbn}\times \PP^{\bbn}$
  respectively.  The sheaf $\widetilde{M}(\bbd)$ is quasi-isomorphic to the
  complex
  $\mathcal F = \mathbf{R}\pi_{1*}^{} \bigl( \pi_{2}^*\widetilde{M}(\bbd)
  \otimes \mathcal K \bigr)$, where
  \[
    \mathcal F_j = \bigoplus_{\abs{\bba} - p = j} H^p\bigl(\PP^{\bbn},
    \widetilde{M}(\bbd) \otimes \Omega^{\bba}_{\PP^{\bbn}}(\bba)\bigr) \otimes
    \cO_{\PP^{\bbn}}^{}(-\bba).
  \]
  Lemma~\ref{lem:noHigherCohomology} says that
  $H^p\bigl(\PP^{\bbn}, \widetilde{M}(\bbd) \otimes
  \Omega^{\bba}_{\PP^{\bbn}}(\bba)\bigr) \neq 0$ only if
  $-\bba \in \Delta_{\abs{\bba} - p} + \NN^r = \Delta_j +\NN^r$.  Since each
  $\mathcal F_j$ is a sum of line bundles, the corresponding $S$-module $F_j$
  is free.  It follows that the complex
  $F \coloneq [ F_0\gets F_1\gets \dotsb \gets F_{\abs{\bbn}} \gets 0 ]$ is a
  virtual resolution of $M(\bbd)$ with the desired form.  Finally, $M$ is
  $B$-saturated so $H_B^0(M) = 0$ and the hypothesis that $\bbd \in \reg M$
  implies that $H_B^1\bigl( M(\bbd) \bigr)_{\bbp} = 0$ for all
  $\bbp \in \NN^r$, so the Hilbert polynomial and Hilbert function of
  $M(\bbd)$ agree on $\NN^r$.
\end{proof}

%%%%%%%%%%%%%%%%%%%%%%%%%%%%%%%%%%%%%%%%%%%%%%%%%%%%%%%%%%%%%%%%%%%%%%%%%%%%%%
\section{Simpler Virtual Resolutions}
\label{sec:winnow}

\noindent
We describe, in this section, an effective method for producing interesting
virtual resolutions of a given $S$-module.  Unlike the previous section, the
free complex is ordinarily not linear nor acyclic.  Our construction depends
on a $B$-saturated module $M$ as well as an element $\bbd \in \reg M$.
Although Theorem~\ref{thm:freeComplex} defines the corresponding virtual
resolution as a subcomplex of a minimal free resolution of $M$,
Algorithm~\ref{alg:computeWinnow} shows that the subcomplex can be assembled
without first computing the entire minimal free resolution.

\begin{theorem}
  \label{thm:freeComplex}
  For a finitely generated $\ZZ^r$-graded $S$-module $M$, consider a minimal
  free resolution $F$ of $M$.  For a degree $\bbd\in\ZZ^r$ and each $i$, let
  $G_i$ be the direct sum of all free summands of $F_i$ whose generator is in
  degree at most $\bbd+\bbn$, and let $\varphi_i$ be the restriction of the
  $i$-th differential of $F$ to $G_i$.
  \begin{enumerate}[\upshape (i)]
  \item For all $i$, we have $\varphi_i(G_i) \subseteq G_{i-1}$ and
    $\varphi_i\circ \varphi_{i+1}=0$, so $G$ forms a free complex.
  \item Up to isomorphism, $G$ depends only on $M$ and $\bbd$.
  \item If $M$ is $B$-saturated and $\bbd \in \reg M$, then $G$ is a virtual
    resolution of $M$.
  \end{enumerate}
\end{theorem}

\noindent
When $M$ is $B$-saturated and $\bbd \in \reg M$, the complex $G$ is the
\define{virtual resolution of the pair $(M,\bbd)$}.

\begin{proof}[Proof of Theorem~\ref{thm:winnow}]
  This theorem is simply a restatement of part~(iii) of
  Theorem~\ref{thm:freeComplex}.
\end{proof}

To illustrate the basic idea behind the proof of
Theorem~\ref{thm:freeComplex}, we revisit our first example.

\begin{example}
  Let $C$ be the hyperelliptic curve in $\PP^1 \times \PP^2$ defined by the
  ideal $I$ in Example~\ref{exa:curveI}.  The free complex in
  \eqref{eqn:aCMcurve} is the virtual resolution of the pair
  $\bigl(S/I, (4,2) \bigr)$ and it is naturally a subcomplex of the minimal
  free resolution \eqref{eqn:resolution curve} of $S/I$. The corresponding
  quotient complex $E$ is
  \begin{align*}
    \begin{matrix} 
      S(-1,-5)^3 \\[-3pt] 
      \oplus     \\[-3pt] 
      S(0,-8)^1
    \end{matrix} 
    & \gets 
    \begin{matrix} 
      S(-2,-5)^6 \\[-3pt] 
      \oplus     \\[-3pt] 
      S(-1,-7)^1 \\[-3pt] 
      \oplus     \\[-3pt]
      S(-1,-8)^2
    \end{matrix}
    \gets 
    \begin{matrix} 
      S(-3,-5)^3 \\[-3pt] 
      \oplus     \\[-3pt] 
      S(-2,-7)^2 \\[-3pt] 
      \oplus     \\[-3pt] 
      S(-2,-8)^1
    \end{matrix}
    \gets S(-3,-7)^1 \gets 0 \, . 
    \intertext{Restricting attention to the terms of degree $(*,-8)$, we have}
    S(0,-8)^1 &\gets S(-1,-8)^2 \gets S(-2,-8)^1 \gets 0,
  \end{align*}
  which looks like a twist of the Koszul complex on $x_{1,0}$ and $x_{1,1}$.
  In fact, the $(*,-8)$, $(*,-7)$, and $(*,-5)$ strands each appear to have
  homology supported on the irrelevant ideal.  This suggests that the complex
  $\widetilde{E}$ is quasi-isomorphic to zero, and that is what we show in the
  proof of Theorem~\ref{thm:freeComplex}.
\end{example}

\begin{lemma}
  \label{lem:beilisonWindow}
  Let $\cE$ be a bounded complex of coherent $\cO_{\PP^{\bbn}}^{}$-modules.
  If $\cE\otimes \cO_{\PP^{\bbn}}^{}(\bbb)$ has no hypercohomology for all
  $\boldzero \leq \bbb \leq \bbn$, then $\cE$ is quasi-isomorphic to $0$.
\end{lemma}

\begin{proof}
  By Theorem~1.1 in \cite{eisenbud-erman-schreyer-tate-products}, any bounded
  complex of coherent $\cO_{\PP^{\bbn}}^{}$-modules is quasi-isomorphic to a
  Beilinson monad whose terms involve the hypercohomology evaluated at the
  line bundles of the form $\boldzero \leq \bbb \leq \bbn$.  The hypothesis on
  vanishing hypercohomology ensures that this Beilinson monad of $\cE$ is the
  $0$ monad, and hence $\cE$ is quasi-isomorphic to $0$. While Theorem~1.1 in
  \cite{eisenbud-erman-schreyer-tate-products} is stated for a sheaf, the
  authors remark in Equation (1) on page 8 that a similar statement holds for
  bounded complexes of coherent sheaves.
\end{proof}

\begin{proof}[Proof of Theorem~\ref{thm:freeComplex}]
  For part~(i), write $F_i = G_i \oplus E_i$ for each $i$. Each generating
  degree $\bbe$ of $E_i$ satisfies $\bbe\not\leq \bbd+\bbn$.  It follows that,
  for degree reasons, there are no nonzero maps from $G_i$ to $E_{i-1}$.  The
  $i$-th differential $\partial_i \colon F_i\to F_{i-1}$ has a block
  decomposition
  \[
    \partial_i = 
    \bbordermatrix{
      & G_i & E_i \cr 
      G_{i-1} &  \varphi_i  &  * \cr 
      E_{i-1} & 0 & * \cr 
    } \, ,
  \]
  so $\varphi_i(G_i) \subseteq G_{i-1}$ and
  $\partial_{i} \circ \partial_{i+1} = 0$ implies that
  $\varphi_{i} \circ \varphi_{i+1} =0$.  As $G$ only depends on $F$ and
  $\bbd$, part~(ii) follows from the fact that the minimal free resolution of
  $M$ is unique up to isomorphism.  For part~(iii), we may without loss of
  generality replace $M$ by $M(\bbd)$ and $\bbd$ by $\boldzero$.  Let $G$ be
  the virtual resolution of the pair $(M, \boldzero)$ and consider the short
  exact sequence of complexes $0 \to G \to F \to E \to 0$.  It suffices to
  show that the complex $\widetilde{E}$ of sheaves is quasi-isomorphic to
  zero.

  Fix some $\bbb$, where $\boldzero \leq \bbb \leq \bbn$. If
  $\bba\not\leq\bbn$, then the line bundle $\cO_{\PP^{\bbn}}^{}(-\bba+\bbb)$
  has no global sections.  It follows that each summand of
  $\widetilde{F}(\bbb)$ with global sections belongs to $\widetilde{G}(\bbb)$.
  If $H^i\bigl(\PP^{\bbn},\widetilde{F}(\bbb)\bigr)$ is the complex obtained
  by applying the functor $\cF \mapsto H^i(\PP^{\bbn}, \cF)$ to the complex
  $\widetilde{F}(\bbb)$, then we have
  $H^0\bigl( \PP^{\bbn}, \widetilde{F}(\bbb) \bigr) = H^0\bigl(\PP^{\bbn},
  \widetilde{G}(\bbb) \bigr)$.  The notation
  $H^i\bigl(\PP^{\bbn}, \widetilde{F}(\bbb) \bigr)$ should not be confused
  with the hypercohomology group
  $\HHH^i\bigl(\PP^{\bbn}, \widetilde{F}(\bbb)\bigr)$, which equals
  $H^i\bigl(\PP^{\bbn}, \widetilde{M}(\bbb) \bigr)$ because $\widetilde{F}$ is
  a locally-free resolution of the sheaf $\widetilde{M}$.  Since
  $\boldzero \in \reg M$ and $\bbb\geq\boldzero$, the Hilbert polynomial and
  Hilbert function of $M$ agree in degree $\bbb$.  Because $F$ is a minimal
  free resolution of $M$, it follows that the strand
  $[F]_{\bbb}\coloneq [(F_0)_{\bbb} \gets (F_1)_{\bbb} \gets \dotsb]$ is
  quasi-isomorphic to $M_{\bbb}$, and hence
  \[
    \HHH^0\bigl(\PP^{\bbn},\widetilde{F}(\bbb)\bigr) = M_{\bbb} \cong
    [F]_{\bbb} = H^0\bigl(\PP^{\bbn},\widetilde{F}(\bbb)\bigr) \cong
    H^0\bigl(\PP^{\bbn},\widetilde{G}(\bbb)\bigr).
  \]
  If the line bundle $\cO_{\PP^{\bbn}}^{}(-\bba+\bbn)$ has global sections,
  then we see that $\cO_{\PP^{\bbn}}^{}(-\bba+\bbb)$ has no higher cohomology.
  Therefore, the only summands in $\widetilde{F}$ that can potentially have
  higher cohomology are those that also appear in $\widetilde{E}$.  Thus, for
  all $i>0$, we have
  $H^i\bigl(\PP^{\bbn},\widetilde{F}(\bbb)\bigr) =
  H^i\bigl(\PP^{\bbn},\widetilde{E}(\bbb)\bigr)$ and
  $H^i\bigl(\PP^{\bbn},\widetilde{G}(\bbb)\bigr) =0$.  It follows that
  $\HHH^0\bigl(\PP^{\bbn},\widetilde{G}(\bbb)\bigr) \cong
  H^0\bigl(\PP^{\bbn},\widetilde{G}(\bbb)\bigr)$ and
  $\HHH^i\bigl(\PP^{\bbn},\widetilde{G}(\bbb)\bigl) = 0$ for all $i>0$.
  Hence, the long exact sequence in hypercohomology yields
  \[
    \HHH^i\bigl(\PP^{\bbn},\widetilde{E}(\bbb)\bigr) = 
    \begin{cases}
      0 & \text{ if $i=0$,} \\
      \HHH^i\bigl(\PP^{\bbn},\widetilde{F}(\bbb)\bigr) & \text{ if $i>0.$}
    \end{cases}
  \]
  Since $\bbb \in \reg M $, the sheaf $\widetilde{M}(\bbb)$ has no higher
  cohomology and $\widetilde{F}(\bbb)$ has no higher hypercohomology.  By
  Lemma~\ref{lem:beilisonWindow}, we conclude that $\widetilde{E}$ is
  quasi-isomorphic to $0$.
\end{proof}

Although Theorem~\ref{thm:freeComplex} presents the virtual resolution of the
pair $(M, \bbd)$ as a subcomplex of a minimal free resolution, the following
algorithm shows that we can compute a virtual resolution of the pair
$(M, \bbd)$ without first computing an entire minimal free resolution.  Our
approach is similar to Theorem~1.5 in \cite{maclagan-smith}, which allows one
to certify that an element belongs to the regularity of a module from just
part of its minimal free resolution.  Alternatively, one can verify that an
element belongs to the regularity by using the Tate resolutions appearing in
Section~4 of \cite{eisenbud-erman-schreyer-tate-products}; the package
\emph{TateOnProducts}~\cite{tateOnProducts} already implements these
algorithms in \emph{Macaulay2}~\cite{M2}.  For a module $M$ and a degree
$\bbd\in\ZZ^r$, let $M_{\leq \bbd}$ denote the submodule generated by
$\bigoplus_{\bba \leq \bbd} M_{\bba}$.

\begin{algorithm}[Computing Virtual Resolutions of a Pair]
  \label{alg:computeWinnow} 
  $\;$

  \begin{tabbing}
    Output:i \= \kill
    Input: \> A finitely generated $\ZZ^r$-graded $B$-saturated $S$-module $M$
    and \\
    \> a vector $\bbd \in \ZZ^r$ such that $\bbd \in \reg M$. \\
    Output: \> The virtual resolution $G$ of the pair $(M, \bbd)$. \\[8pt]
    W \= W \= W \= \kill
    \> Initialize $K \coloneq M$ and $i \coloneq 0$;\\
    \> While $K \neq 0$ do \\
    \> \> Choose a homogeneous minimal set $\mathscr{G}$ of generators for
    $K$;\\
    \> \> Initialize $G_i \coloneq \bigoplus_{g \in \mathscr{G}} S\bigl(- \deg(g)
    \bigr)$ and $\varphi_i \colon G_i \to K$ to be the corresponding
    surjection;\\
    \> \> Set $K \coloneq (\Ker \varphi_i)_{\leq \bbd + \bbn}$;\\
    \> \> Set $i \coloneq i + 1$;\\
    \> Return $G \coloneq [ G_0 \xleftarrow{\;\varphi_1\;} G_1
    \xleftarrow{\;\varphi_2\;} G_2 \longleftarrow \dotsb ]$.
  \end{tabbing}
\end{algorithm}

\begin{proof}[Proof of Correctness]
  Let $G$ be the complex produced by the algorithm, let $F$ be the minimal
  free resolution of $M$, and let $G'$ be the virtual resolution of
  $(M,\bbd)$.  Let $\varphi$, $\partial$, and $\psi$ be the differentials of
  $G$, $\bF$, and $G'$ respectively.  We have $G_0 = F_0 = G'_0$, as
  $\bbd \in \reg M$ implies that $M$ is generated in degree at most $\bbd$ by
  Theorem~1.3 in \cite {maclagan-smith}.

  The definition of $G'$ implies that $ (\Image \partial_i)_{\leq \bbd+\bbn}$
  equals $\Image \psi_i$.  We use induction on $i$ to prove that
  $\Image \varphi_i = (\Image \partial_i)_{\leq \bbd+\bbn}$.  When $i=1$, we
  have $\Image \varphi_1 = (\Image \partial_1)_{\leq \bbd+\bbn}$.  For $i>1$,
  the key observation is
  \[
    (\Image \partial_i)_{\leq \bbd+\bbn} =
    (\Ker \partial_{i-1})_{\leq\bbd+\bbn} =
    \bigl(\Ker \partial_{i-1}|_{G_{i-1}}\bigr)_{\leq\bbd+\bbn},
  \]
  where the righthand equality holds because any element in
  $(\Ker \partial_{i-1})_{\leq \bbd+\bbn}$ only depends on the restriction of
  $\partial_{i-1}$ to $G_{i-1}$.  By induction, we have
  $\Image \varphi_{i-1} = (\Image \partial_{i-1})_{\leq \bbd+\bbn}$, so
  \[
    \bigl(\Ker \partial_{i-1}|_{G_{i-1}}\bigr)_{\leq\bbd+\bbn} = \bigl( \Ker
    \varphi_{i-1}\bigr)_{\leq \bbd+\bbn} = \Image \varphi_i.
  \]
  Therefore, we conclude that $\Image \varphi_i=\Image \psi_i$ for all $i$ and
  $G\cong G'$.
\end{proof}

\begin{remark}
  Although Algorithm~\ref{alg:computeWinnow} bears a similarity with the
  linear resolutions considered in Proposition~2.7 of
  \cite{eisenbud-erman-schreyer-tate-products}, the free modules appearing in
  a given term of our virtual resolutions need not be generated in a single
  degree and our complexes need not be acyclic.
\end{remark}

The subsequent example demonstrates that the virtual resolution of a pair
does depend on the choice of element in the regularity.

\begin{example}
  \label{exa:lotsOfPoints}
  Let $Z \subset \PP^1 \times \PP^1 \times \PP^2$ be the subscheme consisting
  of $6$ general points and let $I$ be the corresponding $B$-saturated
  $S$-ideal.  \emph{Macaulay2}~\cite{M2} shows that the minimal free
  resolution of $S/I$ has the form
  $S^1 \gets S^{37}\gets S^{120}\gets S^{166} \gets S^{120} \gets S^{45} \gets
  S^7 \gets 0$, where for brevity we have omitted the twists.  Using
  Proposition~6.7 in \cite{maclagan-smith}, it follows that, up to symmetry in
  the first two factors, the minimal elements in the regularity of $S/I$ are
  $(5,0,0), (2,1,0)$, $(1,0,1)$, and $(0,0,2)$.  Table~\ref{tab:one} compares
  some basic numerical invariants for the minimal free resolution and the
  corresponding virtual resolutions.  The total Betti numbers of a free
  complex $F$ are the ranks of the terms $F_i$ ignoring the twists.
  \begin{table}[ht]
    \centering
    \caption{Comparison of various free complexes associated to $Z$}
    \label{tab:one}
    \vspace*{-0.75em}
    \renewcommand{\arraystretch}{1.0}
    \begin{tabular}{llc} \hline
      \multicolumn{1}{c}{Type of Free Complex} 
      & \multicolumn{1}{c}{Total Betti Numbers} 
      & \multicolumn{1}{c}{Number of Twists} \\ \hline 
      minimal free resolution of $S/I$ 
      & $(1,37,120,166,120,45,7)$ & $78$ \\ 
      virtual resolution of the pair $\bigl( S/I,(5,0,0) \bigr)$ 
      & $(1, 24, 50, 33, 6) $ & $18$ \\ 
      virtual resolution of the pair $\bigl( S/I, (2,1,0) \bigr)$
      & $(1, 29, 73, 66, 21)$ & $22$ \\ 
      virtual resolution of the pair $\bigl( S/I, (1,0,1) \bigr)$
      & $(1, 25, 63, 57, 18)$ & $15$ \\ 
      virtual resolution of the pair $\bigl( S/I, (0,0,2) \bigr)$
      & $(1, 22, 51, 42, 12)$ & $13$ \\ \hline
    \end{tabular}
    \renewcommand{\arraystretch}{1}
  \end{table}
  Since $Z$ has codimension $4$, part~(i) of
  Proposition~\ref{pro:resolutionBound} implies that any virtual resolution
  for $S/I$ must have length at least $4$, so the minimum is achieved by all
  of these virtual resolutions.  All four virtual resolutions also have a
  nonzero first homology module, which is supported on the irrelevant
  ideal. The first three virtual resolutions also have nonzero second homology
  modules.  By examining the twists, we see that no pair of these virtual
  resolutions are comparable.  This corresponds to the fact that the
  $\reg(S/I)$ has several distinct minimal elements.
\end{example}

%%%%%%%%%%%%%%%%%%%%%%%%%%%%%%%%%%%%%%%%%%%%%%%%%%%%%%%%%%%%%%%%%%%%%%%%%%%%%%
\section{Virtual Resolutions for Punctual Schemes} 
\label{sec:punctual}

\noindent
This section formulates and proves an extension of a property of points in
projective space.  While every punctual scheme in projective space is
arithmetically Cohen--Macaulay, this fails when the ambient space is a product
of projective spaces; the minimal free resolution is nearly always too long.
However, by using virtual resolutions, we obtain a unexpected variant for
points in $\PP^{\bbn}$.

To state this analogue, recall that the irrelevant ideal on $\PP^{\bbn}$ is
$B = \bigcap_{i=1}^r \ideal{x_{i,0}, x_{i,1}, \dotsc, x_{i,n_i}}$.  For a
vector $\bba \in \NN^r$, set
$B^{\bba} \coloneq \bigcap_{i=1}^r \ideal{x_{i,0}, x_{i,1}, \dotsc,
  x_{i,n_i}}^{a_i}$.  With this notation, we may easily choose a different
algebra to represent the structure sheaf on our punctual subscheme.  In
contrast with the virtual resolutions in Section~\ref{sec:winnow}, the next
theorem produces acyclic free complexes.

\begin{theorem}
  \label{thm:hiddenCMPoints}
  If $Z \subset \PP^{\bbn}$ is a zero-dimensional scheme and $I$ is the
  corresponding $B$-saturated $S$-ideal, then there exists $\bba \in \NN^r$
  with $a_r = 0$ such that the minimal free resolution of
  $S/(I \cap B^{\bba})$ has length equal $\abs{\bbn} = \dim \PP^{\bbn}$.
  Moreover, any $\bba \in \NN^r$ with $a_r = 0$ and other entries sufficiently
  positive yields such a virtual resolution of $S/I$.
\end{theorem}

\begin{proof}[Proof of Theorem~\ref{thm:ptsACM}]
  Applying Theorem~\ref{thm:hiddenCMPoints}, it suffices to choose
  $Q = B^{\bba}$ for any $\bba \in \NN^r$ with $a_r = 0$ and other entries
  sufficiently positive.
\end{proof}

While Theorem~\ref{thm:hiddenCMPoints} establishes that, for appropriate
$\bba \in \NN^r$, the projective dimension of $S / (I \cap B^{\bba})$ equals
the codimension of $Z$, this does not mean that the algebra
$S / (I \cap B^{\bba})$ is Cohen--Macaulay; the ideal $I \cap B^{\bba}$ will
often fail to be unmixed.  For instance, on $\PP^2 \times \PP^2$, the ideals
$\ideal{x_{i,0}, x_{i,1}, x_{i,2}}$ for $1 \leq i \leq 2$ have codimension $3$
whereas a zero-dimensional scheme $Z$ would have codimension $4$.
Nevertheless, we do get Cohen--Macaulayness in one case.

\begin{corollary}
  \label{cor:HilbertBurch}
  If $Z \subset \PP^1\times\PP^1$ is a zero-dimensional subscheme and $I$ is
  the corresponding $B$\nobreakdash-saturated $S$-ideal, then there exists an
  ideal $Q$ whose radical is $\ideal{x_{1,0},x_{1,1}}$ such that
  \begin{enumerate}[\upshape (i)]
  \item the algebra $S / (I \cap Q)$ is Cohen--Macaulay, and
  \item there exists an $(m+1) \times m$ matrix over $S$ whose maximal minors
    generate $I \cap Q$.
  \end{enumerate}
\end{corollary}

\begin{proof}
  Theorem~\ref{thm:hiddenCMPoints} yields an $\bba \in \NN^r$ such that
  $I \cap B^{\bba}$ has projective dimension $2$.  On $\PP^ 1\times \PP^1$,
  the irrelevant ideal $B$ also has codimension $2$, so
  $S / (I \cap B^{\bba})$ has codimension $2$.  Thus, the algebra
  $S / (I \cap B^{\bba})$ is Cohen--Macaulay.  The second statement is an
  immediate consequence of the Hilbert--Burch
  Theorem~\cite{eisenbud-book}*{Theorem~20.15}.
\end{proof}

\begin{remark}
  \label{rem:projectiveBundles}
  Although this paper focuses on products of projective spaces, our proofs for
  both Theorem~\ref{thm:hiddenCMPoints} and Corollary~\ref{cor:HilbertBurch}
  can be adapted to hold in the more general context of iterated projective
  bundles.  For instance, let $X$ be the Hirzebruch surface with Cox ring
  $S = \kk[y_0, y_1, y_2, y_3]$ where the variables have degrees $(1,0)$,
  $(1,0)$, $(-2,1)$, and $(0,1)$ respectively.  Let $Z \subset X$ be the
  scheme-theoretic intersection of $y_0^5 y_2^2 + y_1^{} y_3^2$ and
  $y_0^{} y_1^{} + y_2^{} y_1^3$.  If $I$ is the $B$-saturated $S$-ideal of
  $Z$, then $S/I$ has projective dimension $3$ and
  $S/(I\cap \ideal{y_0, y_1}^a)$ has projective dimension $2$ for any
  $a \geq 4$.
\end{remark}

As with the proof of Theorem~\ref{thm:freeComplex}, we collect two lemmas
before proving Theorem~\ref{thm:hiddenCMPoints}.

\begin{lemma}
  \label{lem:regSequence}
  If $Z \subset \PP^{\bbn}$ is a zero-dimensional scheme and $I$ is the
  corresponding $B$-saturated $S$-ideal, then there exists $\bba \in \ZZ^{r}$
  with $a_r = 0$ such that the depth of $(S/I)_{\geq \bba}$ is $r$.  Moreover,
  this holds for any $\bba \in \ZZ^{r}$ with $a_r=0$ and other entries
  sufficiently positive.
\end{lemma}

\begin{proof}
  Extending the ground field does not change the depth of a module, so we
  assume that $\kk$ is an infinite field.  Since
  $\dim (S/I)_{\geq \bba} = \dim(S/I) = r$, the depth of $S/I$ is bounded
  above by $r$.  For each $1\leq i\leq r$, choose a general linear element
  $\ell_i$ in $\ideal{x_{i,0}, x_{i,1}, \dotsc, x_{i,n_i}}$. We claim that the
  elements $\ell_1,\ell_2,\dots,\ell_r$ form a regular sequence on
  $(S/I)_{\geq \bba}$.

  Let $M \coloneq \bigoplus_{\bbb \in \NN^r} H^0\big(Z, \cO_Z^{}(\bbb)\big)$.  By
  construction, the elements $\ell_1,\ell_2,\dots,\ell_r$ form a regular
  sequence on $M$.  Since $I$ is $B$-saturated, it follows that
  $H_B^0(S/I)=0$. The exact sequence relating local cohomology and sheaf
  cohomology~\cite{CLS}*{Theorem~9.5.7} gives
  \[
    \begin{tikzcd}[row sep = 1.25em]
      0 & \bigl( H_B^1(S/I) \bigr)_{\geq 0} \arrow[l] \arrow[d,
      "\cdot \ell_i" right] & M \arrow[l] \arrow[d, "\cdot \ell_i" right] &
      S/I \arrow[l] \arrow[d, "\cdot \ell_i" right ] & 0 \relphantom{\, ,}
      \arrow[l] \\
      0 & \bigl( H_B^1(S/I) \bigr)_{\geq 0} \arrow[l] & M \arrow[l] &
      S/I \arrow[l] & 0 \, . \arrow[l] 
    \end{tikzcd}
  \]
  The middle vertical arrow is an isomorphism because $\ell_i$ does not vanish
  on any point in $Z$.  Hence, the Snake
  Lemma~\cite{eisenbud-book}*{Exercise~A3.10} implies that the right vertical
  arrow is injective.

  Focusing on the last component of $\ZZ^r$, we identify the Cox ring
  $R \coloneq \kk[x_{r,0}, x_{r,1}, \dotsc, x_{r,n_r}]$ of the factor
  $\PP^{n_r}$ with the subring
  $(S)_{(\boldzero,*)}
  \coloneq\bigoplus_{\alpha\in\NN}(S)_{(\boldzero,\alpha)}$ of $S$.  For any
  $\bbc' \in \ZZ^{r-1}$, consider the $R$-module
  $(S/I)_{(\bbc',*)} \coloneq \bigoplus_{\alpha\in\NN} (S)_{(\bbc',\alpha)}$.
  These modules form a directed set: for $\bbc', \bbc''\in \ZZ^{s-1}$ with
  $\bbc''\geq \bbc'$, multiplication by the form
  $\ell_1^{\bbc''_1-\bbc'_1}\ell_2^{\bbc''_2-\bbc'_2}\dotsb
  \ell_{s-1}^{\bbc''_{s-1}-\bbc'_{s-1}}$ gives the inclusion
  $(S/I)_{(\bbc',*)}\subseteq (S/I)_{(\bbc'',*)}$.  Each
  $R$\nobreakdash-module $(S/I)_{(\bbc',*)}$ is a submodule of
  $(M)_{(\bbc',*)}$ and $(M)_{(\bbc',*)} \cong (M)_{(0,*)}$.  It follows that
  the $(S/I)_{(\bbc',*)}$ form an increasing sequence of finitely-generated
  $R$-submodules of $(M)_{(0,*)}$, so this sequence stabilizes.  In
  particular, if $\bba'\in \ZZ^{r-1}$ is sufficiently positive, then the
  inclusion $(S/I)_{(\bbc',*)}\subseteq (S/I)_{(\bbc'+\bbe_i,*)}$ is an
  isomorphism for each $\bbc'\geq \bba'$ and each $1\leq i \leq r-1$.  Hence,
  $\ell_1,\ell_2,\dots,\ell_{r-1}$ form a regular sequence on
  $(S/I)_{\geq (\bba',0)}$ and
  \[
    \frac{(S/I)_{\geq (\bba',0)}}{\ideal{\ell_1, \ell_2, \dotsc, \ell_{r-1}}}
    \cong (S/I)_{(\bba',*)}.
  \]
  Since $I$ is $B$-saturated, $\ell_r$ is regular on $S/I$, so it is also
  regular on the $R$-module $(S/I)_{(\bba',*)}$.  Setting $\bba = (\bba',0)$,
  the Auslander--Buchsbaum Formula~\cite{eisenbud-book}*{Theorem~19.9}
  completes the proof.
\end{proof}

\begin{remark}
  The proof of Lemma~\ref{lem:regSequence} shows that $(S/I)_{\geq \bba}$ has
  a multigraded regular sequence of length $r$, but neither
  $S / (I \cap B^{\bba})$ nor $S/B^{\bba}$ generally has a multigraded regular
  sequence of length $r$.
\end{remark}

\begin{lemma}
  \label{lem:pdimIrrelevant}
  If $\bba \coloneq (a_1, a_2, \dotsc, a_i, 0, \dotsc, 0) \in \ZZ^{r}$ for some
  $1 \leq i \leq r$, then the projective dimension of $S/B^{\bba}$ is at most
  $n_1 + n_2 + \dotsb + n_i + 1$.
  % Moreover, this holds for any $\bba \in \ZZ^{r}$ with $a_j=0$ for
  % $i+1\leq j\leq r$ and other entries sufficiently positive.
\end{lemma}

\begin{proof}
  We proceed by induction on $i$.  The base case $i = 1$ is just the Hilbert
  Syzygy Theorem~\cite{eisenbud-book}*{Theorem~1.3} applied to the Cox ring of
  $\PP^{n_1}$.  When $i > 1$, set
  $\bba' \coloneq (a_1, a_2, \dotsc, a_{i-1}, 0, \dotsc, 0)$ and
  $\bba'' \coloneq (0, \dotsc, 0, a_i, 0, \dotsc, 0)$, so that
  $\bba = \bba' + \bba''$ and $B^{\bba} = B^{\bba'} \cap B^{\bba''}$.  The
  short exact sequence
  $0 \gets S/(B^{\bba'}+B^{\bba''}) \gets S/B^{\bba'} \gets S/B^{\bba} \gets
  0$ yields
  \begin{equation}
    \label{eqn:sesInequality}
    \pdim (S/ B^{\bba}) 
    \leq \max\bigl\{ 
    \pdim (S/B^{\bba'}), 
    \pdim (S/B^{\bba''}), 
    \pdim \bigl( S/(B^{\bba'} + B^{\bba''})\bigr) - 1
    \bigr\}.
  \end{equation}
  By induction, the projective dimension of $S/B^{\bba'}$ is at most
  $n_1 + n_2 + \dotsb + n_{i-1} + 1$.  By the Hilbert Syzygy Theorem on
  $\PP^{n_i}$, the projective dimension of $S/B^{\bba''}$ is at most
  $n_i + 1$.  Since $B^{\bba'}$ and $B^{\bba''}$ are supported on disjoint
  sets of variables,
  $\pdim \bigl (S / (B^{\bba'}+B^{\bba''}) \bigr) = \pdim(S/B^{\bba'}) +
  \pdim(S/B^{\bba''})$ which is at most $n_1 + n_2 + \dotsb + n_i + 2$.
  Applying \eqref{eqn:sesInequality}, we obtain the result.
\end{proof}

\begin{proof}[Proof of Theorem~\ref{thm:hiddenCMPoints}]
  Let $\bba \in\ZZ^r$ with $a_r=0$ and other entries sufficiently
  positive. There is a short exact sequence
  $0 \gets S/B^{\bba} \gets S/(I\cap B^{\bba}) \gets (S/I)_{\geq \bba} \gets
  0$.  By Proposition~\ref{pro:resolutionBound}, it suffices to prove that the
  projective dimension of $S/(I\cap B^{\bba})$ is at most
  $n_1 + n_2 + \dotsb + n_r$.  Lemma~\ref{lem:regSequence} shows that
  $(S/I)_{\geq \bba}$ has projective dimension $n_1 + n_2 + \dotsb + n_r$ and
  Lemma~\ref{lem:pdimIrrelevant} shows that $S/B^{\bba}$ has projective
  dimension at most $n_1 + n_2 + \dotsb + n_{r-1} + 1$.  It follows that the
  projective dimension of $S/(I\cap B^{\bba})$ is at most
  $n_1 + n_2 + \dotsb + n_r$ as well.
\end{proof}

The next example compares the virtual resolutions produced by
Theorem~\ref{thm:winnow} and Theorem~\ref{thm:ptsACM}; neither seems to have a
definitive advantage over the other.

\begin{example}
  As in Example~\ref{exa:lotsOfPoints}, let
  $Z \subset \PP^1 \times \PP^1 \times \PP^2$ be the subscheme consisting of
  $6$ general points and let $I$ be the corresponding $B$-saturated $S$-ideal.
  Table~\ref{tab:two} compares some basic numerical invariants for virtual
  \begin{table}[ht]
    \centering
    \caption{Comparison of various free complexes associated to $Z$}
    \label{tab:two}
    \vspace*{-0.75em}
    \renewcommand{\arraystretch}{1.0}
    \begin{tabular}{llc} \hline
      \multicolumn{1}{c}{Type of Free Complex} 
      & \multicolumn{1}{c}{Total Betti Numbers} 
      & \multicolumn{1}{c}{Number of Twists} \\ \hline 
      minimal free resolution of $S/I$ 
      & $(1,37,120,166,120,45,7)$ & $78$ \\ 
      virtual resolution from $\bba = (2,1,0)$
      & $(1, 17, 34, 24, 6) $ & $12$ \\ 
      virtual resolution from $\bba = (3,3,0)$
      & $(1, 22, 42, 27, 6)$ & $13$ \\ \hline
    \end{tabular}
    \renewcommand{\arraystretch}{1}
  \end{table}
  resolutions arising from Theorem~\ref{thm:hiddenCMPoints}.  Since the
  virtual resolutions in Table~\ref{tab:two} involve non-minimal generators
  for $I$, they are different than those in Table~\ref{tab:one}.  Conversely,
  the virtual resolutions appearing in Table~\ref{tab:one} cannot be obtained
  from Theorem~\ref{thm:ptsACM} because those free complexes are not acyclic.
\end{example}

We end this section by extending Corollary~\ref{cor:HilbertBurch} to any
smooth projective toric surface.

\begin{proposition}
  \label{pro:pointsInGeneralPosition}
  Fix a smooth projective toric surface $X$. Let $Z \subset X$ be the
  subscheme consisting of $m$ general points and let $I$ be the corresponding
  $B$-saturated $S$-ideal.  There exists a virtual resolution
  $F \coloneq [S \gets F_1 \xleftarrow{\;\;\varphi\;\;} F_2 \gets 0]$ of $S/I$
  such that $\rank(F_1) = \rank(F_2) + 1$, the maximal minors of $\varphi$
  generate an ideal $J$ with $I = (J : B^\infty)$, and $S/J$ is a
  Cohen--Macaulay algebra.
\end{proposition}

\begin{proof}
  Any smooth projective toric surface can be realized as a blowup of
  $\pi \colon X \to Y$, where $Y$ is $\PP^2$ or a Hirzebruch
  surface~\cite{CLS}*{Theorem~10.4.3}.  Since $\pi(Z) \subset Y$ is a punctual
  scheme, we can apply the Hilbert--Burch Theorem when $Y$ is $\PP^2$, or
  Corollary~\ref{cor:HilbertBurch} and Remark~\ref{rem:projectiveBundles} when
  $Y$ is a Hirzebruch surface, to obtain a resolution of $\cO_{\pi(Z)}^{}$ of
  the form $ \cO_Y^{} \gets \cE_1 \gets \cE_2\gets 0, $ where $\cE_1$ and
  $\cE_2$ are sums of line bundles on $Y$.  Our genericity hypothesis implies
  that $Z$ does not intersect the exceptional locus of $\pi$, so
  $\pi^*\cO_Y^{} \gets \pi^* \cE_1 \gets \pi^* \cE_2\gets 0$ is a locally-free
  resolution of $\cO_Z^{}$.  The corresponding complex of $S$-modules is a
  virtual resolution for $S/I$ of the appropriate form.
\end{proof}

\begin{example}
  Consider the del Pezzo surface $X$ of degree $7$ or equivalently the smooth
  Fano toric surface obtained by blowing-up the projective plane at two
  torus-fixed points.  The Cox ring of $X$ is
  $S \coloneq \kk[y_0, y_1, \dotsc, y_4]$ equipped with the $\ZZ^3$-grading
  induced by
  \begin{xalignat*}{6}
    \deg(y_0) &\coloneq \left[ \begin{smallmatrix}
        1 \\ 0 \\ 0 
      \end{smallmatrix} \right] , &
    \deg(y_1) &\coloneq \left[ \begin{smallmatrix}
        -1 \\ \relphantom{-}1 \\ \relphantom{-}0 
      \end{smallmatrix} \right] ,  &
    \deg(y_2) &\coloneq \left[ \begin{smallmatrix}
        \relphantom{-}1 \\ -1 \\ \relphantom{-}1 
      \end{smallmatrix} \right] , &
    \deg(y_3) &\coloneq \left[ \begin{smallmatrix}
        \relphantom{-}0 \\ \relphantom{-}1 \\ -1 
      \end{smallmatrix} \right] , &
   \deg(y_4) &\coloneq \left[ \begin{smallmatrix}
        0 \\ 0 \\ 1 
      \end{smallmatrix} \right] .
  \end{xalignat*}
  With this choice of basis for $\Pic(X)$, the nef cone equals the positive
  orthant.  Let $Z \subset X$ be the subscheme consisting of the three points
  $[1:1:1:1:1]$, $[2:1:3:1:5]$, and $[7:1:11:1:13]$ (expressed in Cox
  coordinates) and let $I$ be the corresponding $B$-saturated $S$-ideal.
  \emph{Macaulay2}~\cite{M2} shows that the minimal free resolution of $S/I$
  has the form
  \begin{align*}
    S^1 \gets
    \begin{matrix} 
      S(-1,0,-1)^1 \\[-3pt] 
      \oplus       \\[-3pt] 
      S(0,-1,-1)^1 \\[-3pt] 
      \oplus       \\[-3pt] 
      S(-1,-1,0)^1 \\[-3pt] 
      \oplus       \\[-3pt]
      S(0,0,-3)^1  \\[-3pt] 
      \oplus       \\[-3pt]  
      S(-3,0,0)^1
    \end{matrix} 
    &\gets 
    \begin{matrix} 
      S(0,-2,-2)^1  \\[-3pt] 
      \oplus        \\[-3pt] 
      S(-1,-1,-1)^2 \\[-3pt] 
      \oplus        \\[-3pt] 
      S(-2,-1,0)^1  \\[-3pt] 
      \oplus        \\[-3pt]
      S(-1,0,-3)^1  \\[-3pt] 
      \oplus        \\[-3pt]  
      S(-3,0,-1)^1
    \end{matrix} 
    \gets 
    \begin{matrix} 
      S(-1,-1,-2)^1 \\[-3pt] 
      \oplus        \\[-3pt] 
      S(-2,-1,-1)^1
    \end{matrix} 
    \gets 0 \, . 
    \intertext{However, there is a virtual resolution of $S/I$ having the form}
    S^1 \gets \,\,\, S(0,-2,0)^3 \,\,\, &\gets \,\,\, S(0,-3,0)^2 \,\,\, \gets
                                          0 \, . \qedhere
  \end{align*}
\end{example}

\begin{example}
  \label{exa:completeIntersectionP1^2}
  Let $Z \subset \PP^1 \times \PP^1$ be the subscheme consisting of $m$
  general points and let $I$ be the corresponding $B$-saturated $S$-ideal.
  Not only is there a Hilbert--Burch-type virtual resolution of $S/I$, it can
  be chosen to be a Koszul complex.  Since
  $\dim H^0\bigl( \PP^1 \times \PP^1, \cO_{\PP^1 \times \PP^1}(i,j) \bigr) =
  (i+1)(j+1)$, the generality of the points implies that
  $\dim H^0\bigl( \PP^1 \times \PP^1, \cO_{Z}(i,j) \bigr) = \min\{(i+1)(j+1),
  m\}$.  Hence, if $m = 2k$ for some $k \in \NN$, then two independent global
  sections of $\cO_{\PP^1\times\PP^1}^{}(1, k)$ vanish on $Z$. Using this
  pair, we obtain a virtual resolution of $S/I$ of the form
  $S \gets S(-1,-k)^2 \gets S(-2,-2k) \gets 0$.  On the other hand, if
  $m = 2k + 1$, then there are independent global sections of
  $\cO_{\PP^1\times\PP^1}^{}(1,k)$ and $\cO_{\PP^1\times\PP^1}^{}(1,k+1)$ that
  vanish on $Z$, so we obtain a virtual resolution of $S/I$ having the form
  \[
  S \gets 
  \begin{matrix} 
    S(-1,-k) \\[-3pt] 
    \oplus \\[-3pt] 
    S(-1,-k-1)
  \end{matrix} 
  \gets S(-2, -2k-1) \gets 0 \, . \qedhere
  \]
\end{example}

\begin{question}
  Does Proposition~\ref{pro:pointsInGeneralPosition} hold for any punctual
  scheme $Z$ in a smooth toric surface  $X$?
\end{question}

%%%%%%%%%%%%%%%%%%%%%%%%%%%%%%%%%%%%%%%%%%%%%%%%%%%%%%%%%%%%%%%%%%%%%%%%%%%%%%
\section{Geometric Applications}
\label{sec:applications}

\noindent
In this section, we showcase four geometric applications of virtual
resolutions.  In particular, each of these support our overarching thesis that
replacing minimal free resolutions by virtual resolutions yields the best
geometric results for subschemes of $\PP^{\bbn}$.

%%----------------------------------------------------------------------------
\subsection*{Unmixedness}

Given a subscheme that has a virtual resolution whose length equals its
codimension, we prove an unmixedness result.  Closely related to
Proposition~\ref{pro:resolutionBound}, this extends the classical unmixedness
result for arithmetically Cohen--Macaulay subschemes, see Corollary~18.14 in
\cite{eisenbud-book}.

\begin{proposition}[Unmixedness]
  \label{pro:unmixedness}
  Let $Z \subset \PP^{\bbn}$ be a closed subscheme of codimension $c$ and let
  $I$ be the corresponding $B$-saturated $S$-ideal.  If $S/I$ has a virtual
  resolution of length $c$, then every associated prime of $I$ has codimension
  $c$.
\end{proposition}

\begin{proof}
  Let $Q$ be an associated prime of $I$ and let $F$ denote a virtual
  resolution of $S/I$ having length $c$. Our hypothesis on $Z$ implies that
  $\codim Q \geq c$.  Since $Q$ does not contain the irrelevant ideal $B$,
  localizing at $Q$ annihilates the homology of $F$ that is supported at $B$.
  Thus, the complex $F_Q$ is a free $S_Q$-resolution of $(S/I)_{Q}$. Since the
  projective dimension of a module is at least its
  codimension~\cite{eisenbud-book}*{Proposition~18.2}, it follows that
  $ c \geq \codim_{S_Q} (S/I)_Q = \codim Q\geq c$.
\end{proof}

%%----------------------------------------------------------------------------
\subsection*{Deformation Theory}  

Using virtual resolutions, we generalize results about unobstructed
deformations for arithmetically Cohen--Macaulay subschemes of codimension two,
arithmetically Gorenstein subschemes of codimension three, and complete
intersections.  To accomplish this, we first observe that the
Piene--Schlessinger Comparison Theorem~\cite{piene-schlessinger} applies more
generally by relating the deformations of a closed subscheme
$Y \subseteq \PP^{\bbn}$ with deformations of a corresponding graded module
over the Cox ring.

\begin{theorem}[Comparison Theorem]
  \label{thm:comparison}
  Let $Y \subset \PP^{\bbn}$ be a closed subscheme and let $I$ be a
  homogeneous $S$-ideal defining $Y$ scheme-theoretically and generated in
  degrees $\bbd_1, \bbd_2, \dotsc, \bbd_s$.  If the natural map
  $(S/I)_{\bbd_i} \to H^0\bigl( Y,\cO_Y^{}(\bbd_i) \bigr)$ is an isomorphism
  for all $1 \leq i \leq s$, then the embedded deformation theory of
  $Y \subset \PP^{\bbn}$ is equivalent to the degree zero embedded deformation
  theory of $\variety(I) \subset \Spec(S)$.
\end{theorem}

\begin{proof}
  Piene and Schlessinger's proof of the Comparison
  Theorem~\cite{piene-schlessinger} goes through essentially verbatim by
  replacing projective space and its coordinate ring with $\PP^{\bbn}$ and its
  Cox ring $S$.
\end{proof}

\begin{proof}[Proof of Theorem~\ref{thm:unobstructedDefs}]
  If $\bbe \in \reg(S/I)$ and $\bF$ is the virtual resolution of the pair
  $(S/I,\bbe)$, then we have $\HH_0(F) = S/J$ for some ideal $J$ whose
  $B$-saturation equals $I$.  By Theorem~\ref{thm:winnow}, the generating
  degrees for $J$ are a subset of those for $I$.  It follows that
  $(S/J)_{\bbd} = H^0\bigl( Y,\cO_Y^{}(\bbd) \bigr)$ for each degree $\bbd$ of
  a generator for $J$.  Therefore, Theorem~\ref{thm:comparison} implies that
  the embedded deformation theory of $Y$ is equivalent to the degree zero
  embedded deformation theory of the subscheme $\variety(J) \subset \Spec(S)$.
  \begin{enumerate}[\upshape (i)]
  \item The virtual resolution $\bF$ has length $2$, so
    Proposition~\ref{pro:resolutionBound} implies that $\bF$ is the minimal
    free resolution of $S/J$.  Thus, $S/J$ is Cohen--Macaulay of codimension
    $2$, and \cite{artin}*{\S5} or \cite{schaps} implies that its embedded
    deformations are unobstructed.
  \item The virtual resolution $\bF$ has length $3$ and
    $\min\{ n_i + 1 : 1 \leq i \leq r \} \geq 3$, so
    Proposition~\ref{pro:resolutionBound} implies that $\bF$ is the minimal
    free resolution of $S/J$.  Thus, $S/J$ is Gorenstein of codimension $3$,
    and Theorem~2.1 in \cite{miro-roig} implies that its embedded deformations
    are unobstructed.
  \item Let $c \coloneq \codim Y$.  As $F$ is a Koszul complex, we have
    $F_1 = \bigoplus_{i=1}^c S(-\bbd_i)$.  Since $F$ is a virtual resolution
    of $S/I$, we also see that $\widetilde{J}$ equals the ideal sheaf
    $\mathcal{I}_Y$ for $Y \subset \PP^{\bbn}$.  The complex
    $\widetilde{F_1} \gets \widetilde{F_2} \gets \widetilde{F_3} \gets \dotsb$
    is a locally-free resolution of $\mathcal{I}_Y$, so the normal bundle of
    $Y$ in $\PP^{\bbn}$ is
    $\mathcal{N}_{Y/\PP^{\bbn}} \coloneq \cHom(\mathcal{I}_Y /
    \mathcal{I}_Y^2, \cO_Y^{}) \cong \bigoplus_{i=1}^c \cO_Y^{}(\bbd_i)$.  For
    a fixed deformation of $Y$, the obstruction is a \v{C}ech cocycle in
    $H^1(Y, \mathcal{N}_{Y/\PP^{\bbn}})$ determined by local lifts of the
    syzygies; see Theorem~6.2 in \cite{hartshorne-deformation}.  However,
    since $Y$ is a scheme-theoretic complete intersection, its syzygies are
    all Koszul, so we can define this cocycle by lifting those Koszul syzygies
    globally on $Y$.  Hence, the \v{C}ech cocycle in
    $H^1(Y, \mathcal{N}_{Y/\PP^{\bbn}})$ is actually a coboundary and the
    obstruction vanishes.\qedhere
  \end{enumerate}
\end{proof}
 
\begin{remark}
  In part~(ii) of Theorem~\ref{thm:unobstructedDefs}, we suspect that the
  hypothesis $\min\{n_i\}\geq 2$ is unnecessary.
\end{remark}

\begin{example}
  Consider the hyperelliptic curve $C \subset \PP^1 \times \PP^2$ defined in
  Example~\ref{exa:curveI}.  Applying part~(i) of
  Theorem~\ref{thm:unobstructedDefs}, the virtual resolution from
  \eqref{eqn:aCMcurve} implies that $C$ has unobstructed embedded
  deformations.  Alternatively, this curve has a virtual resolution
  \[
    S \xleftarrow{\;\; \left[
        \begin{smallmatrix}
          f & g
        \end{smallmatrix}
      \right] \;\;} 
    \begin{matrix} 
      S(-2,-2) \\[-3pt] 
      \oplus   \\[-3pt] 
      S(-3,-1) 
    \end{matrix} 
    \xleftarrow{\;\; \left[
        \begin{smallmatrix}
          -g \\ \phantom{-}f
        \end{smallmatrix} \right] \;\;}
    S(-5,-3) \longleftarrow 0, 
  \]
  where
  $f = x_{1,0}^2 x_{2,0}^2 + x_{1,1}^2 x_{2,1}^2 + x_{1,0}^{} x_{1,1}^{}
  x_{2,2}^2$ and
  $g = x_{1,0}^3 x_{2,2}^{} + x_{1,1}^3 ( x_{2,0}^{} + x_{2,1}^{})$, so
  part~(iii) of Theorem~\ref{thm:unobstructedDefs} provides another proof that
  this curve has unobstructed embedded deformations.
\end{example}

%%----------------------------------------------------------------------------
\subsection*{Regularity of Tensor Products}

Using virtual resolutions, we can prove bounds for the regularity of a tensor
product, similar to the bounds obtained for projective space; see
Proposition~1.8.8 in \cite{lazarsfeld}. Let $\bbe_i$ denote the $i$-th
standard basis vector in $\ZZ^r = \Pic(\PP^{\bbn})$.

\begin{proposition}
  \label{pro:regularityTensor}
  Let $\cE$ and $\cF$ be coherent $\cO_{\PP^{\bbn}}^{}$-modules such that
  ${\mathcal{T}\!\!or}^j(\cE, \cF) = 0$ for all $j > 0$.  If
  $\bba \in \reg \cE$ and $\bbb \in \reg \cF$, then we have
  $\bba + \bbb + \boldone - \bbe_i \in \reg(\cE\otimes \cF) $ for each
  $1 \leq i \leq r$.
\end{proposition}

\begin{proof}
  Let $M$ and $N$ be the $B$-saturated $S$-modules corresponding to $\cE$ and
  $\cF$.  Since $M(\bba)$ is $\boldzero$\nobreakdash-regular,
  Theorem~\ref{thm:linearResolutions} implies that it has a virtual resolution
  $F_0\gets F_1\gets \dotsb$, where the degree of each generator of $F_i$
  belongs to $\Delta_i+\NN^r$.  Similarly, $N(\bbb)$ has a virtual resolution
  $G_0 \gets G_1\gets \dotsb$ satisfying the same conditions.  The vanishing
  of Tor-groups implies that $H \coloneq F\otimes G$ is a virtual resolution of
  $M(\bba) \otimes N(\bbb)$.  Since
  $\Delta_i + \Delta_j + \boldone - \bbe_i\subseteq \Delta_{i+j}+\NN^r$, it
  follows that the degree of each generator of the free module
  $H(\boldone - \bbe_i)_k$ belongs to $\Delta_k+\NN^r$, for each $k$.  Hence
  Theorem~\ref{thm:linearResolutions} implies that
  $\bigl(M(\bba)\otimes N(\bbb)\bigr)(\boldone - \bbe_i)$ is
  $\boldzero$-regular.
\end{proof}

\begin{remark}
  Proposition~\ref{pro:regularityTensor} is sharp.  When $r=1$, it recovers
  Proposition~1.8.8 in \cite{lazarsfeld}, as the higher Tor-groups vanish
  whenever one of the two sheaves is locally free.  If $r > 1$, then it is
  possible to have $\boldzero$-regular sheaves whose tensor product is not
  $\boldzero$-regular.  For instance, if $D, {D'} \subset \PP^1 \times \PP^2$
  are degree $(1,1)$-hypersurfaces, then the product
  $\cO_D^{} \otimes \cO_{D'}^{}$ is isomorphic to the structure sheaf
  $\cO_C^{}$ for a curve $C$ with $H^1\bigl(C, \cO_C^{}(0,-1) \bigr) \neq 0$.
\end{remark}

%%----------------------------------------------------------------------------
\subsection*{Vanishing of Higher Direct Images}

A relative notion of Castelnuovo--Mumford regularity with respect to a given
morphism is defined in terms of the vanishing of derived pushforwards; see
Example~1.8.24 in \cite{lazarsfeld}.  Just as virtual resolutions yield
sharper bounds on multigraded Castelnuovo--Mumford regularity, they also
provide sharper bounds for the vanishing of derived pushforwards.  For some
$1 \leq s \leq r$, fix a subset
$\{ i_1, i_2, \dotsc, i_s \} \subseteq \{1, 2, \dotsc, r \}$ and let $Y$ be
the corresponding product
$\PP^{n_{i_1}} \times \PP^{n_{i_2}} \times \dotsb \times \PP^{n_{i_r}}$ of
projective spaces.  The canonical projection $\pi \colon \PP^{\bbn} \to Y$
induces an inclusion $\pi^* \colon \Pic(Y) \to \Pic(\PP^{\bbn})$ and we write
$\rho \colon \Pic(\PP^{\bbn}) \to \coker(\pi^*) = \Pic(\PP^{\bbn}) / \Pic(Y)$.

\begin{proposition}
  \label{pro:pushforward}
  Let $M$ be a finitely-generated $\ZZ^r$-graded $S$-module and consider
  $\bba \in \ZZ^{r}$.  If we have $\rho(\bba) \in \rho(\reg M)$, then it
  follows that $\mathbf{R}^i \pi_{*} \widetilde{M}(\bba) = 0$ for all $i > 0$.
\end{proposition}

\begin{proof}
  Since $\rho(\bba)\in \rho(\reg M)$, we can choose $\bbb \in \Pic(Y)$ such
  that $\bba + \pi^* \bbb \in \reg M$.  The Projection
  Formula~\cite{Hartshorne}*{Exercise~III.8.3} gives
  $\mathbf{R}^i \pi_{*} \widetilde{M}(\bba) \otimes \cO_Y^{}(\bbb) =
  \mathbf{R}^i \pi_{*} \widetilde{M} (\bba+\pi^* \bbb)$ for all
  $\bbb \in \Pic(Y)$.  Hence, by replacing $\bba$ with $\bba+\pi^*\bbb$, we
  assume that $\bba$ itself lies in $\reg M$.

  Let $G$ be the virtual resolution of the pair $(M,\bba)$, and consider a
  summand $S(-\bbc)$ of $G$.  By definition, we have $\bbc \leq \bba + \bbn$.
  It follows that $-\bbn\leq -\bbc+\bba$ and
  $\mathbf{R}^i \pi_{*} \cO_{\PP^{\bbn}}^{}(-\bbc +\bba)= 0$ for all $i > 0$.
  From the hypercohomology spectral sequence
  $E_2^{p,q} \coloneq \mathbf{R}^p \widetilde{G_{-q}} \Longrightarrow
  \mathbf{R}^{p+q} \widetilde{M}$, we conclude that the higher direct images
  of $\widetilde{M}(\bba)$ also vanish.
\end{proof}

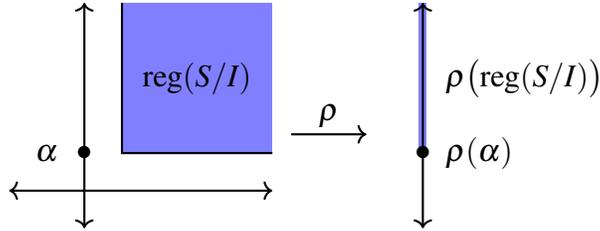
\begin{figure}[ht]
  \begin{tikzpicture}[scale=0.45]
    \fill[fill=blue,semitransparent] (8.9,1) -- (9.1,1) -- (9.1,5) -- (8.9,5);
    \fill[fill=blue, semitransparent] (1,4.99) -- (1,1) -- (4.99,1) --
    (4.99,4.99);
    \draw[thick, <->] (-2,0) -- (5,0);
    \draw[thick, <->] (0,-1) -- (0,5);
    \draw[thick,-] (1,5) -- (1,1) -- (5,1);

    \draw (-1,1) node {$\alpha$};
    \draw (0,1) node {$\bullet$};
    \draw (3,3) node {$\reg(S/I)$};
    % \fill[pattern=dots, semitransparent] (0,0)--(.4,0)--(2.16,3.14)--(2,3.4);
    \draw[thick, ->] (5.5,1.5) -- (7.5,1.5);
    \draw (6.5,2.0) node {$\rho$};
    \draw[thick, <->] (9,-1) -- ( 9,5);
    \draw (10.5,1) node {$\rho(\alpha)$};
    \draw (11.7,3) node {$\rho\bigl( \reg(S/I) \bigr)$};
    \draw (9,1) node {$\bullet$};
  \end{tikzpicture}
  \caption{Representation of $\reg(S/I)$ and its image under $\rho$.}
  \label{fig:pushforward}
\end{figure}

\begin{example}
  \label{exa:surface}
  Let $Y\subset \PP^1\times\PP^3$ be the surface defined by the $B$-saturated ideal 
  \[
    I = \left\langle
      \begin{array}{>{\raggedright\footnotesize}p{400pt}}
        $x_{1,1}^{} x_{2,1}^{2} + x_{1,0}^{} x_{2,0}^{} x_{2,2}^{} + x_{1,1}^{}
        x_{2,1}^{} x_{2,3}^{}$, $x_{1,0}^{2} x_{2,1}^{2} + x_{1,0}^{}
        x_{1,1}^{} x_{2,2}^{2} + x_{1,1}^{2} x_{2,0}^{} x_{2,3}^{}$, 
        $x_{1,0}^{} x_{2,1}^{4} - x_{1,0}^{} x_{2,0}^{} x_{2,2}^{3} +$
        $\relphantom{+} x_{1,0}^{} x_{2,1}^{3} x_{2,3}^{} - x_{1,1}^{} x_{2,0}^{2} x_{2,2}^{}
        x_{2,3}^{}$, 
        $x_{2,1}^{6} - x_{2,0}^{} x_{2,1}^{2} x_{2,2}^{3} + 2 x_{2,1}^{5}
        x_{2,3}^{} + x_{2,0}^{3} x_{2,2}^{2} x_{2,3}^{} - x_{2,0}^{}
        x_{2,1}^{} x_{2,2}^{3} x_{2,3}^{} + x_{2,1}^{4} x_{2,3}^{2}$
      \end{array}
    \right\rangle \, , 
  \]
  and let $\pi \colon \PP^1\times\PP^3\to\PP^1$ be the projection onto the
  first factor.  To understand the vanishing of the higher direct images of
  $\cO_Y^{}$, we consider the minimal free resolution of $S/I$ which has the
  form
  \[
    S^1 \gets 
    \begin{matrix} 
      S(-2,-2)^1 \\[-3pt] 
      \oplus     \\[-3pt] 
      S(-1,-2)^1 \\[-3pt] 
      \oplus     \\[-3pt] 
      S(-1,-4)^1 \\[-3pt] 
      \oplus     \\[-3pt] 
      S(0,-6)^1 
    \end{matrix} 
    \gets 
    \begin{matrix} 
      S(-2,-4)^2 \\[-3pt] 
      \oplus     \\[-3pt] 
      S(-1,-6)^2
    \end{matrix}
    \gets S(-2,-6)^1 \gets 0. 
  \]
  If we tensor the corresponding locally-free resolution with the line bundle
  $\cO_Y^{}(0,3)$, then none of the terms in the resulting complex have
  nonzero higher direct images, so $\mathbf{R}^1 \pi_* \cO_Y^{}(0,c) = 0$ for
  $c \geq 3$.  However, Proposition~\ref{pro:pushforward} yields a sharper
  vanishing result.  Since \emph{Macaulay2}~\cite{M2} shows that
  $(2,1) \in \reg(S/I)$, we have $\mathbf{R}^1 \pi_*\cO_Y^{}(0,c) = 0$ for all
  $c \geq 1$. This bound is sharp because a general fiber of $\pi$ is a curve
  of genus $1$.
\end{example}

%%%%%%%%%%%%%%%%%%%%%%%%%%%%%%%%%%%%%%%%%%%%%%%%%%%%%%%%%%%%%%%%%%%%%%%%%%%%%%
\section{Questions}
\label{sec:open}

\noindent 
We expect that virtual resolutions will produce further analogues of theorems
involving minimal free resolutions on projective space.  We close by
highlighting several promising directions.

The first question is to find a notion of depth that controls the minimal
length of a virtual resolution and provides an analogue of the
Auslander--Buchsbaum Theorem.

\begin{question}
  Given an $S$-module $M$, what invariants of $M$ determine the length of the
  shortest possible virtual resolution of $M$?
\end{question}

\noindent
Even understanding this question for curves in $\PP^1\times\PP^2$ would be
compelling.  In light of Theorem~\ref{thm:unobstructedDefs}, this case would
produce unirationality results for certain parameter spaces of curves.

\begin{question}
  For which values of $d$, $e$, and $g$, does there exist a smooth curve in
  $\PP^1\times\PP^2$ of bidegree $(d,e)$ and genus $g$ with a virtual resolution
  of the form $S\gets F_1\gets F_2\gets 0$?
\end{question}

Proposition~\ref{pro:unmixedness} and Theorem~\ref{thm:unobstructedDefs}
suggest that having a virtual resolution whose length equals the codimension of
the underlying variety can have significant geometric implications. As these
results parallel the arithmetic Cohen--Macaulay property over projective
space, it would be interesting to seek out analogues of being arithmetically
Gorenstein.

\begin{question}
  Consider a positive-dimensional subscheme $Z\subseteq \PP^{\bbn}$ such that
  $\omega_Z=\cO_{\PP^{\bbn}}^{}(\bbd)|_{Z}$ for some $\bbd\in \ZZ^r$.  Is
  there a self-dual virtual resolution of $Z$?
\end{question}

\noindent 
It would also be interesting to better understand scheme-theoretic complete
intersections.

\begin{question}
  Develop an algorithm to determine if a subvariety $Z\subseteq \PP^{\bbn}$
  has a virtual resolution that is a Koszul complex.  This is already
  interesting in the case of points on $\PP^1\times\PP^1$.
\end{question}

Finally, we believe that many of these results should hold for more general
toric varieties.

\begin{question}
  Prove an analogue of Proposition~\ref{pro:splendidComplexesExist} for an
  arbitrary smooth toric variety.
\end{question} 

%%%%%%%%%%%%%%%%%%%%%%%%%%%%%%%%%%%%%%%%%%%%%%%%%%%%%%%%%%%%%%%%%%%%%%%%%%%%%%

\raggedright
	

\begin{bibdiv}
  \begin{biblist}

    \bib{artin}{book}{
      author={Artin, Michael},
      title={\href{http://www.math.tifr.res.in/~publ/ln/tifr54.pdf}%
        {On deformations of singularities}},
      publisher={Tata Institute of Fundamental Research, Bombay},
      date={1976},
    }

    \bib{audin}{book}{
      author={Audin, Mich{\`e}le},
      title={\href{http://dx.doi.org/10.1007/978-3-0348-7221-8}%
        {The topology of torus actions on symplectic manifolds}},
      series={Progress in Mathematics~93},
      % volume={93},
      % note={Translated from the French by the author},
      publisher={Birkh\"auser Verlag, Basel},
      date={1991},
      pages={181},
      % isbn={3-7643-2602-6},
      % review={\MR{1106194}},
      % doi={10.1007/978-3-0348-7221-8},
    }

    \bib{bayer-sturmfels}{article}{
      author={Bayer, Dave},
      author={Sturmfels, Bernd},
      title={\href{http://dx.doi.org/10.1515/crll.1998.083}%
        {Cellular resolutions of monomial modules}},
      journal={J. Reine Angew. Math.},
      volume={502},
      date={1998},
      pages={123--140},
      issn={0075-4102},
      % review={\MR{1647559}},
      % doi={10.1515/crll.1998.083},
    }

    \bib{BC}{article}{
      author={Botbol, Nicol\'as},
      author={Chardin, Marc},
      title={\href{http://dx.doi.org/10.1016/j.jalgebra.2016.11.017}%
        {Castelnuovo Mumford regularity with respect to multigraded ideals}},
      journal={J. Algebra},
      volume={474},
      date={2017},
      pages={361--392},
      % issn={0021-8693},
      % review={\MR{3595796}},
      % doi={10.1016/j.jalgebra.2016.11.017},
    }
    
    \bib{buchsbaum-eisenbud-gor3}{article}{
      author={Buchsbaum, David A.},
      author={Eisenbud, David},
      title={\href{http://dx.doi.org/10.2307/2373926}%{
        {Algebra structures for finite free resolutions, and some structure
          theorems for ideals of codimension $3$}},
      journal={Amer. J. Math.},
      volume={99},
      date={1977},
      number={3},
      pages={447--485},
      % issn={0002-9327},
      % review={\MR{0453723}},
    }
	
    \bib{cox}{article}{
      author={Cox, David A.},
      title={\href{https://arxiv.org/abs/alg-geom/9210008}%
        {The homogeneous coordinate ring of a toric variety}},
      journal={J. Algebraic Geom.},
      volume={4},
      date={1995},
      number={1},
      pages={17--50},
      % issn={1056-3911},
      % review={\MR{1299003}},
    }

    \bib{caldararu}{article}{
      author={C{\u{a}}ld{\u{a}}raru, Andrei},
      title={\href{http://dx.doi.org/10.1090/conm/388/07256}%
        {Derived categories of sheaves: a skimming}},
      conference={
        title={Snowbird lectures in algebraic geometry},
      },
      book={
        series={Contemp. Math.~388},
        % volume={388},
        publisher={Amer. Math. Soc., Providence, RI},
      },
      date={2005},
      pages={43--75},
      % review={\MR{2182889}},
      % doi={10.1090/conm/388/07256},
    }
    
    \bib{CLS}{book}{
      author={Cox, David A.},
      author={Little, John B.},
      author={Schenck, Henry K.},
      title={\href{http://dx.doi.org/10.1090/gsm/124}%
        {Toric varieties}},
      series={Graduate Studies in Mathematics~124},
      % volume={124},
      publisher={Amer. Math. Soc., Providence, RI},
      date={2011},
      pages={xxiv+841},
      % isbn={978-0-8218-4819-7},
      % review={\MR{2810322}},
      % doi={10.1090/gsm/124},
    }

    \bib{DS}{article}{
      author={Decker, Wolfram},
      author={Schreyer, Frank-Olaf},
      title={\href{http://dx.doi.org/10.1006/jsco.1999.0323}%
        {Non-general type surfaces in ${\bf P}^4$: some remarks on bounds and
          constructions}},
      % note={Symbolic computation in algebra, analysis, and geometry
      % (Berkeley, CA, 1998)},
      journal={J. Symbolic Comput.},
      volume={29},
      date={2000},
      number={4-5},
      pages={545--582},
      % issn={0747-7171},
      % review={\MR{1769655}},
      % doi={10.1006/jsco.1999.0323},
    }
    
    \bib{DFS}{article}{
      author={Deopurkar, Anand},
      author={Fedorchuk, Maksym},
      author={Swinarski, David},
      title={\href{http://dx.doi.org/10.14231/AG-2016-001}%
        {Toward GIT stability of syzygies of canonical curves}},
      journal={Algebr. Geom.},
      volume={3},
      date={2016},
      number={1},
      pages={1--22},
      % issn={2214-2584},
      % review={\MR{3455418}},
      % doi={10.14231/AG-2016-001},
    }
    
    \bib{EGA3.1}{article}{
      label={EGA3I},
      author={Grothendieck, Alexander},
      author={Dieudonn\'e, Jean A.},
      title={\href{http://www.numdam.org/item?id=PMIHES_1961__11__5_0}%
        {\'El\'ements de g\'eom\'etrie alg\'ebrique. III. \'Etude
          cohomologique des faisceaux coh\'erents, Premi\`ere partie}},
      journal={Inst. Hautes \'Etudes Sci. Publ. Math.},
      number={11},
      date={1961},
      pages={5--167},
      % issn={0073-8301},
      % review={\MR{0217085}},
    }
    
    \bib{EGA3.2}{article}{
      label={EGA3II},
      author={Grothendieck, Alexander},
      author={Dieudonn\'e, Jean A.},
      title={\href{http://www.numdam.org/item?id=PMIHES_1963__17__5_0}%
        {\'El\'ements de g\'eom\'etrie alg\'ebrique. III. \'Etude
          cohomologique des faisceaux coh\'erents, Seconde partie}},
      % language={French},
      journal={Inst. Hautes \'Etudes Sci. Publ. Math.},
      number={17},
      date={1963},
      pages={5--91},
      % issn={0073-8301},
      % review={\MR{0163911}},
    }
    
    \bib{EL}{article}{
      author={Ein, Lawrence},
      author={Lazarsfeld, Robert},
      title={\href{http://dx.doi.org/10.1007/s10240-015-0072-2}%
        {The gonality conjecture on syzygies of algebraic curves of large
          degree}},
      journal={Publ. Math. Inst. Hautes \'Etudes Sci.},
      volume={122},
      date={2015},
      pages={301--313},
      issn={0073-8301},
      % review={\MR{3415069}},
      % doi={10.1007/s10240-015-0072-2},
    }
    
    \bib{ein-laz-1993}{article}{
      author={Ein, Lawrence},
      author={Lazarsfeld, Robert},
      title={\href{http://dx.doi.org/10.1007/BF01231279}%
        {Syzygies and Koszul cohomology of smooth projective varieties of
          arbitrary dimension}},
      journal={Invent. Math.},
      volume={111},
      date={1993},
      number={1},
      pages={51--67},
      % issn={0020-9910},
      % review={\MR{1193597}},
      % doi={10.1007/BF01231279},
    }
    
    \bib{eisenbud-book}{book}{
      author={Eisenbud, David},
      title={\href{http://dx.doi.org/10.1007/978-1-4612-5350-1}%
        {Commutative algebra with a view toward algebraic geometry}},
      series={Graduate Texts in Mathematics~150},
      % volume={150},
      % note={With a view toward algebraic geometry},
      publisher={Springer-Verlag, New York},
      date={1995},
      pages={xvi+785},
      % isbn={0-387-94268-8},
      % isbn={0-387-94269-6},
      % review={\MR{1322960}},
      % doi={10.1007/978-1-4612-5350-1},
    }
    
    \bib{eisenbud-erman-schreyer-tate-products}{article}{
      author={Eisenbud, David},
      author={Erman, Daniel},
      author={Schreyer, Frank-Olaf},
      title={\href{http://dx.doi.org/10.1007/s40306-015-0126-z}%
        {Tate resolutions for products of projective spaces}},
      journal={Acta Math. Vietnam.},
      volume={40},
      date={2015},
      number={1},
      pages={5--36},
      % issn={0251-4184},
      % review={\MR{3331930}},
      % doi={10.1007/s40306-015-0126-z},
    }

    \bib{tateOnProducts}{misc}{
      author={Eisenbud, David},
      author={Erman, Daniel},
      author={Schreyer, Frank-Olaf},
      author={Stillman, Michael~E.},
      title={TateOnProducts},
      date={2018},
      publisher={a \emph{Macaulay2} package available at
        \url{http://www.math.uiuc.edu/Macaulay2/}},
    }
        
    \bib{EP}{article}{
      author={Eisenbud, David},
      author={Popescu, Sorin},
      title={\href{http://dx.doi.org/10.1007/s002220050315}%
        {Gale duality and free resolutions of ideals of points}},
      journal={Invent. Math.},
      volume={136},
      date={1999},
      number={2},
      pages={419--449},
      issn={0020-9910},
      % review={\MR{1688433}},
      % doi={10.1007/s002220050315},
    }
    
    \bib{farkas}{article}{
      author={Farkas, Gavril},
      title={\href{http://dx.doi.org/10.1353/ajm.0.0053}%
        {Koszul divisors on moduli spaces of curves}},
      journal={Amer. J. Math.},
      volume={131},
      date={2009},
      number={3},
      pages={819--867},
      issn={0002-9327},
      % review={\MR{2530855}},
      % doi={10.1353/ajm.0.0053},
    }
    
    \bib{fujita-semipositive}{article}{
      author={Fujita, Takao},
      title={\href{http://dx.doi.org/10.1007/BFb0099977}%
        {Vanishing theorems for semipositive line bundles}},
      conference={
        title={Algebraic geometry},
        address={Tokyo/Kyoto},
        date={1982},
      },
      book={
        series={Lecture Notes in Math.~1016},
        % volume={1016},
        publisher={Springer, Berlin},
      },
      date={1983},
      pages={519--528},
      % review={\MR{726440}},
      % doi={10.1007/BFb0099977},
    }
    
    \bib{almost-ring-theory}{book}{
      author={Gabber, Ofer},
      author={Ramero, Lorenzo},
      title={\href{http://dx.doi.org/10.1007/b10047}%
        {Almost ring theory}},
      series={Lecture Notes in Math.~1800},
      % volume={1800},
      publisher={Springer-Verlag, Berlin Heidelberg},
      date={2003},
      pages={vi+307},
      % isbn={3-540-40594-1},
      % review={\MR{2004652}},
      % doi={10.1007/b10047},
    }

    \bib{GP}{article}{
      author={Gallego, Francisco J.},
      author={Purnaprajna, Bangere P.},
      title={\href{http://dx.doi.org/10.1515/crll.1999.506.145}%
        {Projective normality and syzygies of algebraic surfaces}},
      journal={J. Reine Angew. Math.},
      volume={506},
      date={1999},
      pages={145--180},
      % issn={0075-4102},
      % review={\MR{1665689}},
      % doi={10.1515/crll.1999.506.145},
    }
    
    \bib{GGP}{article}{
      author={Geramita, Anthony V.},
      author={Gimigliano, Alessandro},
      author={Pitteloud, Yves},
      title={\href{http://dx.doi.org/10.1007/BF01446634}%
        {Graded Betti numbers of some embedded rational $n$-folds}},
      journal={Math. Ann.},
      volume={301},
      date={1995},
      number={2},
      pages={363--380},
      % issn={0025-5831},
      % review={\MR{1314592}},
      % doi={10.1007/BF01446634},
    }

    \bib{green-laz-1986}{article}{
      author={Green, Mark},
      author={Lazarsfeld, Robert},
      title={\href{http://dx.doi.org/10.1007/BF01388754}%
        {On the projective normality of complete linear series on an algebraic
          curve}},
      journal={Invent. Math.},
      volume={83},
      date={1986},
      number={1},
      pages={73--90},
      % issn={0020-9910},
      % review={\MR{813583}},
      % doi={10.1007/BF01388754},
    }
    
    \bib{GV}{book}{
      author={Guardo, Elena},
      author={Van Tuyl, Adam},
      title={\href{http://dx.doi.org/10.1007/978-3-319-24166-1}%
        {Arithmetically Cohen-Macaulay sets of points in $\mathbb{P}^{1}
          \times \mathbb{P}^{1}$}},
      series={SpringerBriefs in Mathematics},
      publisher={Springer, Cham},
      date={2015},
      pages={viii+134},
      % isbn={978-3-319-24164-7},
      % isbn={978-3-319-24166-1},
      % review={\MR{3443335}},
    }
    
    \bib{Hartshorne}{book}{
      author={Hartshorne, Robin},
      title={\href{http://dx.doi.org/10.1007/978-1-4757-3849-0}%
        {Algebraic geometry}},
      series={Graduate Texts in Mathematics~52},
      publisher={Springer-Verlag, New York-Heidelberg},
      date={1977},
      pages={xvi+496},
      % isbn={0-387-90244-9},
      % review={\MR{0463157}},
    }
    
    \bib{hartshorne-deformation}{book}{
      author={Hartshorne, Robin},
      title={\href{http://dx.doi.org/10.1007/978-1-4419-1596-2}%
        {Deformation theory}},
      series={Graduate Texts in Mathematics~257},
      % volume={257},
      publisher={Springer, New York},
      date={2010},
      pages={viii+234},
      % isbn={978-1-4419-1595-5},
      % review={\MR{2583634}},
      % doi={10.1007/978-1-4419-1596-2},
    }
    
    \bib{phantom}{article}{
      author={Hochster, Melvin},
      author={Huneke, Craig},
      title={\href{http://dx.doi.org/10.1090/memo/0490}%
        {Phantom homology}},
      journal={Mem. Amer. Math. Soc.},
      volume={103},
      date={1993},
      number={490},
      pages={vi+91},
      % issn={0065-9266},
      % review={\MR{1144758}},
      % doi={10.1090/memo/0490},
    }
    
    \bib{huybrechts}{book}{
      author={Huybrechts, Daniel},
      title={\href{http://dx.doi.org/10.1093/acprof:oso/9780199296866.001.0001}%
        {Fourier-Mukai transforms in algebraic geometry}},
      series={Oxford Mathematical Monographs},
      publisher={The Clarendon Press, Oxford University Press, Oxford},
      date={2006},
      pages={viii+307},
      isbn={978-0-19-929686-6},
      isbn={0-19-929686-3},
      % review={\MR{2244106}},
      % doi={10.1093/acprof:oso/9780199296866.001.0001},
    }
    
    \bib{lazarsfeld}{book}{
      author={Lazarsfeld, Robert},
      title={\href{http://dx.doi.org/10.1007/978-3-642-18808-4}%
        {Positivity in algebraic geometry. I}},
      series={Series of Modern Surveys in Mathematics~48},
      % Ergebnisse der Mathematik und ihrer Grenzgebiete. 3. Folge. A
      % Series of Modern Surveys in Mathematics [Results in Mathematics and
      % Related Areas. 3rd Series. A Series of Modern Surveys in Mathematics]},
      % volume={48},
      % note={Classical setting: line bundles and linear series},
      publisher={Springer-Verlag, Berlin},
      date={2004},
      pages={xviii+387},
      % isbn={3-540-22533-1},
      % review={\MR{2095471}},
      % doi={10.1007/978-3-642-18808-4},
    }
    
    \bib{maclagan-smith}{article}{
      author={Maclagan, Diane},
      author={Smith, Gregory G.},
      title={\href{http://dx.doi.org/10.1515/crll.2004.040}%
        {Multigraded Castelnuovo-Mumford regularity}},
      journal={J. Reine Angew. Math.},
      volume={571},
      date={2004},
      pages={179--212},
      % issn={0075-4102},
      % review={\MR{2070149}},
      % doi={10.1515/crll.2004.040},
    }

    \bib{M2}{misc}{
      label={M2},
      author={Grayson, Daniel~R.},
      author={Stillman, Michael~E.},
      title={Macaulay2, a software system for research
        in algebraic geometry},
      publisher = {available at \url{http://www.math.uiuc.edu/Macaulay2/}},
    }
    
    \bib{miro-roig}{article}{
      author={Mir{\'o}-Roig, Rosa M.},
      title={Nonobstructedness of Gorenstein subschemes of codimension $3$ in
        ${\bf P}^n$},
      journal={Beitr\"age Algebra Geom.},
      number={33},
      date={1992},
      pages={131--138},
      % issn={0138-4821},
      % review={\MR{1163647}},
    }
    
    \bib{piene-schlessinger}{article}{
      author={Piene, Ragni},
      author={Schlessinger, Michael},
      title={\href{http://dx.doi.org/10.2307/2374355}%
        {On the Hilbert scheme compactification of the space of twisted
          cubics}},
      journal={Amer. J. Math.},
      volume={107},
      date={1985},
      number={4},
      pages={761--774},
      % issn={0002-9327},
      % review={\MR{796901}},
      % doi={10.2307/2374355},
    }

    \bib{Roberts}{article}{
      author={Roberts, Paul},
      title={Le th\'eor\`eme d'intersection},
      % language={French, with English summary},
      journal={C. R. Acad. Sci. Paris S\'er. I Math.},
      volume={304},
      date={1987},
      number={7},
      pages={177--180},
      % issn={0249-6291},
      % review={\MR{880574}},
    }
    
    \bib{schaps}{article}{
      author={Schaps, Mary},
      title={\href{http://dx.doi.org/10.2307/2373859}%
        {Deformations of Cohen--Macaulay schemes of codimension $2$ and
          non-singular deformations of space curves}},      
      journal={Amer. J. Math.},
      volume={99},
      date={1977},
      number={4},
      pages={669--685},
      % issn={0002-9327},
      % review={\MR{0491715}},
    }
    
    \bib{V}{article}{
      author={Voisin, Claire},
      title={\href{http://dx.doi.org/10.1007/s100970200042}%
        {Green's generic syzygy conjecture for curves of even genus lying on a
          $K3$ surface}},
      journal={J. Eur. Math. Soc. (JEMS)},
      volume={4},
      date={2002},
      number={4},
      pages={363--404},
      issn={1435-9855},
      % review={\MR{1941089}},
      % doi={10.1007/s100970200042},
    }
    
  \end{biblist}
\end{bibdiv}
\end{document}